\documentclass[11pt, leqno]{amsart}
\usepackage{etoolbox}
\makeatletter
\patchcmd{\@maketitle}
  {\ifx\@empty\@dedicatory}
  {\ifx\@empty\@date \else {\vskip3ex \centering\footnotesize\@date\par\vskip1ex}\fi
   \ifx\@empty\@dedicatory}
  {}{}
\patchcmd{\@adminfootnotes}
  {\ifx\@empty\@date\else \@footnotetext{\@setdate}\fi}
  {}{}{}
\makeatother
\usepackage[utf8]{inputenc}
\usepackage[T1]{fontenc}
\usepackage[all]{xy}
\usepackage{t1enc}
\usepackage{amsfonts,amsmath,amssymb,amsthm,color,amsbsy}

\usepackage{eucal}
\usepackage{enumerate}
\usepackage{graphicx}
\usepackage{pstricks-add}
\usepackage{pst-plot}
\usepackage{float}
\usepackage{pst-node}
\usepackage{url}
\usepackage{paralist}
\usepackage{mathrsfs}
\usepackage{colortbl}
\usepackage{datetime}

\newtheorem{thm}{Theorem}[section]
\newtheorem{lem}[thm]{Lemma}
\newtheorem{cor}[thm]{Corollary}
\newtheorem{prop}[thm]{Proposition}
\theoremstyle{definition}

\newtheorem{rem}[thm]{Remark}

\def\restr{\!\!\upharpoonright_}
\def\Gr{\mathrm{Gr}}
\def\m{\multimap}
\def\f{\varphi}
\def\E{\mathbb{E}}

\def\e{\varepsilon}
\def\pr{\right )}
\def\le{\left (}
\def\K{\mathcal{K}}
\def\cl{\mathrm{cl}\,}
\def\conv{\mathrm{conv}}
\def\Ind{\mathrm{Ind}}

\def\R{\mathbb{R}}
\def\mm{\kern +2pt \raisebox{+0.5 pt}{$\shortmid$}\kern -2pt\hbox{$\multimap$}\kern +2pt}

\title[Krasnosels'kii formula]{The constrained Krasnosels'kii formula for parabolic differential inclusions}

\author{Wojciech Kryszewski and Jakub Siemianowski}
\address{Faculty of Mathematics and Computer Sciences, Nicolaus Copernicus University}
\email{wkrysz@mat.umk.pl, jsiem@mat.umk.pl}



\begin{document}

\baselineskip=14pt

\mbox{}

\maketitle

\vspace{-8mm}

\begin{abstract}
We consider a constrained evolution inclusions of  parabolic type \eqref{inkluzja-rozn} involving an $m$-dissipative linear operator  and the source term of multivalued type in a Banach space and topological properties of the solution map.
We show a relation between the constrained fixed point index of the Krasnosel'skii--Poincar\'{e} operator of translation along trajectories associated with \eqref{inkluzja-rozn} and the appropriately defined constrained degree of $A + F\le 0 , \cdot \pr $ of the right-hand side in \eqref{inkluzja-rozn}. Our results extend those of \cite{cw} and \cite{gab-krysz}.
\end{abstract}

\section{Introduction}

We study the initial value problem for a  semilinear differential inclusion
\begin{equation}\label{inkluzja-rozn}
\begin{cases}
\dot{u}\le t \pr \in A  u\le t \pr  + F \le t , u\le t \pr \pr ,\quad t\in [0,1],\; u \in \K,\\
u\le 0 \pr =x \in \K,
\end{cases}
\end{equation}
where $\E$ is a Banach space, $\K\subset \E$ is a closed convex set of state constraints, $A:D \le A \pr \subset \E \to \E $ generates a compact strongly continuous linear semigroup $\{S(t)\}_{t\geq 0}$ on $\E$ and $F:[0,1]\times \K \m \E$ is a set-valued map. A continuous $u:[0,1]\to\E$ is a (mild) {\em solution} to  \eqref{inkluzja-rozn} if it stays in $\K$, i.e., $u(t)\in \K$ and
$$u(t)=S(t)x+\int_0^tS(t-s)w(s)\,ds$$
for all $t\in [0,1]$, where $w\in L^1([0,1],\E)$ and $w(s)\in F(s,u(s))$ a.e. on $[0,1]$.\\
\indent The study of \eqref{inkluzja-rozn} is justified and motivated by a partial differential inclusion of parabolic type
\[\begin{cases}
u_t  - \Delta u \in \f \le t,x, u \pr, \quad t\in [0,1],\, x\in \Omega,\, u \in K\\
u(0, \cdot ) = g=(g_1,...,g_N) \in L^2 \le \Omega ,\R^N\pr ,\; g\le x\pr \in K \text{ for a.e. }x\in \Omega,\\
u| _{[0,1] \times \partial \Omega} = 0,
\end{cases}\leqno\le\ast\pr\
\]
where $\Omega \subset \R^M$ is a bounded domain with smooth boundary, $K\subset \R^N$ is convex closed and $\f:[0,1]\times \Omega \times K\m \R^N$ is upper semicontinuous with convex compact values. Generalized systems of the form $\le\ast\pr$ model reaction-diffusion processes with uncertain reaction term or (via set-valued regularization) those with  discontinuous reaction term. We are looking for (strong) solutions with values in $K$, i.e. $u = \le u_1, \ldots ,u_N \pr : [0,1]\times \Omega \to\R^N$ such that $u(t,x)\in K$ for all $t\in [0,1]$, $x\in\Omega$, $u_i(t,\cdot)\in H^1\cap H^2(\Omega)$ for a.a. $t\in [0,1]$, the function $t\mapsto h_i(t):=\Delta u_i(t,\cdot)$ belongs to $L^1([0,1],L^2(\Omega))$ and $u_i(t,\cdot)=g+\int_0^t(h(s)+w_i(s))\,ds$ for all $i=1,...,N$, where $w=(w_1,...,w_N):[0,1]\to L^2(\Omega,\R^N)$ is integrable and $w(s)(x)\in\f(s,x,u(s,x))$ a.e. for $s\in [0,1]$ and $x\in\Omega$. The role of the constraining set $K$ may be explained as follows: treating  $u_i$ as the concentration of the $i$-th among $N$ components under diffusion, one has  $u_i \geq 0$  since concentration cannot be negative. On the other hand, there is an upper bound, say $u_i(t,x)\leq R_i$ on $[0,1]\times\Omega$, beyond which the $i$-th component is saturated. Thus, the natural question is to study the existence and behavior of solutions  $u= \le u_1,\ldots , u_N \pr$ in the cube $[0,R_1] \times \ldots \times [0,R_N]$.
This is just a heuristic simplification, and so, instead of the cube, we consider an arbitrary closed convex set $K$.

In order to get solutions to \eqref{inkluzja-rozn} we will rely on the semigroup invariance of $\K$ and the {\em weak tangency} condition:
\begin{equation}\label{tangency}
F\le t ,x \pr \cap T_\K \le x \pr \neq \emptyset\text{ for all }t\in [0,1], x\in \K,
\end{equation}
where
\begin{equation}\label{def tang}
T_\K \le y\pr := \mathrm{cl} \bigcup_{h>0} h^{-1}\le \K -x \pr=\left\{v\in\E\;\middle\vert\;\lim_{h\to 0^+}\frac{1}{h}d(y+hv,\K)=0\right\}
\end{equation}
stands for the {\em tangent cone} to $\K$ at $y\in \K$ ($\cl$ stands for the closure and $d(z,\K)$ is the distance of $z\in\E$ to $\K$). These conditions, being in fact too strong for the existence only, are very well-justified and, moreover, imply the $R_\delta$-structure of the set of all solutions to \eqref{inkluzja-rozn} and allow to compare the fixed point index of the Poincar\'{e} $t$-operator $\Sigma_t :\K \m \K$, $t>0$, associated with \eqref{inkluzja-rozn} given by
$$\Sigma_t\le x \pr:= \left \{ u \le t\pr \in \K \mid u \text{ is a solution of }\le \ast \pr,\;u \le 0 \pr = x\right \},\;x\in \K,$$
with the below introduced constrained topological degree of the right-hand side $A+F\le 0, \cdot \pr$. In this way we obtain a generalization of the celebrated Krasnosel'skii formula.\\
\indent Recall that the classical Krasnosel'skii formula concerns an ODE
$\dot{x} = f \le t, x \pr$, $x\in \R^N$, $t\in [0,1]$, with locally Lipschitz  $f:[0,1]\times \R^N \to \R^N$, admitting global solutions.
If $U\subset \R^N$ is open bounded and $f(x,0)\neq 0$ for $x$ in the boundary $\partial U$ of $U$, then the Brouwer degrees $\mathrm{deg}_B \le -f\le 0, \cdot \pr , U\pr=\mathrm{deg}_B \le I - P_t , U\pr$, where  $P_t$ is the associated Poincar\'{e} operator (cf. \cite[Lem. 13.1., 13.2.]{krasnoselskii}).  An infinite dimensional variant of the Krasnosel'skii formula was obtained in \cite{cw} in the case of  \eqref{inkluzja-rozn} with single-valued, time-independent and locally Lipschitz nonlinearity $F$ and in the context of bifurcation results
in \cite{gab-krysz}, where the unconstrained situation was considered.

After this introduction the paper is organized as follows: in the second section we introduce the notation along with some auxiliary lemmata; in the third one we discuss in detail  assumptions on $A$, $\K$ and $F$ in \eqref{inkluzja-rozn} and show that they are motivated and follow directly from the natural and mild hypotheses concerning $\le\ast\pr$.
In the fourth section we establish the $R_\delta$-structure of solutions to \eqref{inkluzja-rozn} and, in the fifth one the appropriate degree of the right-hand side in \eqref{inkluzja-rozn} is defined. In the final, sixth section we prove the announced Krasnosel'skii formula.

\section{Preliminaries}
In what follows $\le \E , \| \cdot \| \pr $ denotes a real Banach space, while $\E^\ast$ is the  normed topological dual of $\E$; we write $\langle x,p \rangle $ instead of $p\le x\pr $ for $ x\in \E, p\in \E^\ast$; $\mathcal{L}\le E\pr $ denotes the space of bounded linear operators on $\E$. By $L^1([0,T],\E)$ (resp. ${\mathcal C}([0,1],\E)$) we denote the space of Bochner integrable (resp. continuous) functions $u:[0,T]\to\E$. Recall that $A\subset L^1 \le [0,1], \E \pr$ is \emph{integrably bounded}\label{integrably-bounded} if there exists $\lambda \in L^1\le [0,1], \R\pr$ such that $\| \alpha \le t\pr \| \leq \lambda \le t \pr$ a.e. for every  $\alpha \in A$. If $X$ is a metric space, $\e>0$ then  $B_X(A,\e):=\{x\in X\mid d\le x;A\pr := \inf _{a\in A} d\le x,a\pr<\e\}$. If $X\subset \E$, $Y$ is a topological space, then a continuous $f\colon X\to Y$ is {\em compact} or \emph{completely continuous} if $f\le B \pr $ is relatively compact for each bounded $B\subset X$.

A {\em set-valued} map $ \f : X \m Y$ assigns to each $x\in X$ a nonempty subset $\f \le x \pr \subset Y$.
If $X, Y$ are topological spaces, then $\f$ is \emph{upper semicontinuous} or usc (resp. \emph{lower semicontinuous} or lsc) if $\f^{-1} \le A\pr := \left \{ x\in X \mid \f \le x \pr \cap A \neq \emptyset \right \}$ is closed (resp.  open) for every closed $A\subset Y$.
If $X\subset \E$, then $\f :X\m Y$ is \emph{compact} if it is usc and $\f \le B \pr := \bigcup_{x\in B} \f \le x \pr $ is relatively compact for any bounded $B\subset X$. If $X, Y$ are metric spaces, then $\f:X \m Y$ is \emph{H-usc} (resp. {\em H-lsc})  if for any $x_0\in X$ and $\e>0$ there is $\delta >0$ such that  $\f \le x \pr \subset B_Y \le \f \le x_0 \pr , \e \pr $ (resp. $\f \le x_0 \pr \subset  B_Y\le \f \le x\pr,\e\pr$) for $x\in B_X\le x_0, \delta \pr $ (see \cite{gorn}, \cite{winter} for details and examples concerning set-valued maps).

We present two results that will be frequently used in a form adopted for our needs. The first one is a simple modification of \cite[Lem. 17.]{bader-krysz}.
\begin{lem}\label{selekcja-styczna}
Let $\K\subset\E$ be closed convex, $F:[0,1] \times \K \m \E$ be tangent to $\K$ (see \eqref{tangency}) and H-usc with convex values. For any continuous  $\alpha: [0,1] \times \K \to \le 0, \infty \pr $ there is a locally Lipschitz  $f :[0,1] \times \K \to \E$ such that $f \le t ,x \pr \in T_ \K\le x \pr$ and
\[
f\le t ,x \pr \in F \le B_{[0,1]} \le t, \alpha \le t, x \pr \pr \times B_\K \le x, \alpha \le t ,x \pr \pr \pr + B _\E \le 0, \alpha \le t,x \pr \pr \quad \text{for } t \in [0,1],\; t\in \K.
\eqno\square\]
\end{lem}

If $\le S, \mathcal{F} \pr$ is a measure space, $X$ is a Polish space and $Y$ is a topological space, then
$\f:S\times X \m Y$ is said to be {\em product measurable} if, for every open $U\subset Y$, $\f ^{-1}\le U \pr$ belongs the product $\sigma$-algebra $\mathcal{F}\otimes \mathfrak{B}\le X \pr$, where $\mathfrak{B} \le X \pr $ is the Borel $\sigma$-algebra in $X$.

\begin{thm}\cite[Th. 3.2]{nguyen}\label{nguyen}
Let $\E$ be a separable Banach space, $\K\subset \E$ be closed convex amd
let $F,G:[0,1] \times \K \m \E$ be product measurable (on $[0,1]$ the Lebesgue $\sigma$-algebra is considered)
with closed convex values and such that $F(t,x) \cap G(t,x)\neq\emptyset$ for all $(t,x)$. If  $F \le t, \cdot \pr$ is H-usc and $G \le t, \cdot \pr$ is lsc, for $t\in [0,1]$, then for every $\e >0$ there is a Carath\'eodory map $f:[0,1]\times \K \to \E$  (i.e. $f\le t,\cdot \pr$ is continuous for every $t\in [0,1]$ and $f\le \cdot , x \pr$ is measurable for every $x\in \K$) such that
$$f\le t,x \pr \in G\le t,x \pr\;\;\hbox{and}\;\;f\le t,x \pr \in F \le \left \{ t \right \} \times B_\K \le x,\e \pr \pr + B_\E \le 0 , \e \pr $$
for all $t\in [0,1]$ and $x\in \K$.\hfill$\square$
\end{thm}

\section{From the system of PDE's to an abstract problem}

Let us make the following standing assumptions with respect to \eqref{inkluzja-rozn}:

$\le A \pr $ $A : D \le A \pr \to \E$ generates a {\em compact} $C_0$-semigroup $\left \{ S \le t \pr \right \} _{t\geq 0}$ of linear operators on $\E$;

$\le K \pr $ A closed convex $\K\subset\E$ is {\em semigroup invariant}, i.e. $S\le t \pr \le \K \pr \subset \K$ for every $t\geq 0$;

$\le F_1\pr$ $F : [0,1] \times \K \m \E$ has convex weakly compact values;

$\le F_2 \pr $ \parbox[t]{0.8\textwidth}{ $F$ is product measurable and for any $t\in [0,1]$, the map $\K \ni x \mm F \le t, x \pr \subset \E$ is H-usc;}

$\le F_3\pr$ there is $c>0$ such that $\sup _{y\in F\le t, x\pr} \|y\| \leq c \le 1 + \|x\| \pr \;\text{ for all }\; t\in [0,1],\; x\in \K$;

$\le F_4 \pr $ $F$ is {\em tangent} to $\K$, i.e. $F \le t,x \pr \cap T_\K \le x \pr \neq \emptyset$ for all  $t\in [0,1]$ and $x\in\K$ (see \eqref{def tang}).

We shall show that these assumptions are consistent with hypotheses usually made with respect to the system $\le\ast\pr$. But first let us collect some comments.
\begin{rem}\label{slabo-usc}
(a) In view of (A) there are  $M \geq 1$, $\omega \in \R$ such that $\|S \le t \pr \| \leq M \e ^ {\omega t}$ for $t\geq 0$. For $h>0$ and $h\omega<1$, the {\em resolvent} $J_h := \le \text{I} - h A \pr ^{-1} : \E \to D \le A \pr \subset \E$ is well-defined, belongs to $\mathcal{L}\le E\pr $ and $\| J_h \| \leq M\le 1 - h\omega \pr ^{-1}$ (cf. \cite{pazy}). By \cite[Th. 2.3.3]{pazy},  $\left \{ S\le t\pr \right\}_{t\geq 0}$ is \emph{compact} (i.e. for any $t >0$ an operator $S\le t\pr \in \mathcal{L}\le \E \pr $ is compact) if and only if it is {\em resolvent compact}, i.e.  for $h>0 , h\omega  <1$,  $J_h$ is compact and $\le 0, +\infty \pr \ni t \mapsto S \le t \pr \in \mathcal{L} \le \E \pr $ is continuous.\\
\indent (b)  Assumption $\le K\pr$ means that if the reaction term $F$ vanishes, then the diffusion process $u(t)=S(t)x$, $x\in \K$, survives in $\K$. It holds if and only if $J_h \le \K \pr \subset \K$ for $h>0$ with  $h\omega <1$ (comp. \cite[Sec. 3.1.]{KanKry} and cf. \cite[Rem. 4.6.]{cw}).\\
\indent (c) Assumptions $\le F_1 \pr$-$\le F_2 \pr $ together with \cite[Proposition 2.3]{bothe} imply that for all $t\in [0,1]$ the set-valued map $F \le  t ,\cdot \pr :  \K \m \E$ is {\em weakly usc}, i.e. usc with respect to the original topology in $\K$ and the weak topology in $\E$. In particular, for each $t\in [0,1]$ the image $F\le \{t\}\times D\pr \subset \E$ of a compact subset $D\subset  \K $ is weakly compact. Moreover, since values of $F$ are convex, we gather that the graph of $F(t,\cdot)$ is closed in $\K\times\E$, where the original topology in $\K$ and the weak topology in $\E$ are considered, i.e., if $x_n\to x$ in $\K$, $y_n\in F(t,x_n)$ and $y_n\rightharpoonup y$ (weakly), then $y\in F(t,x)$. Condition $\le F_3\pr$ implies the global (unconstrained) existence of solutions.\\
\indent (d) It is easy to see that $\le K\pr$ implies that for all $x\in\K$
$$T_\K(x) \subset T_\K^A(x):=\left\{v\in\E\;\middle\vert\;\liminf_{t\to 0^+}\frac{1}{t}d(S(t)x+tv,\K)=0\right\}.$$
Hence $(K)$ together with $(F_4)$ imply that
\begin{equation}\label{Astycz}F(t,x)\cap T_\K^A(x)\neq\emptyset,\;\;x\in\K,\; t\in [0,1].\end{equation}
The sets $T_\K^A(x)$, $x\in \K$, have been introduced by Pavel \cite{Pavel} and condition \eqref{Astycz} shown to be necessary and sufficient for the existence of (mild) solutions surviving in $\K$ of \eqref{inkluzja-rozn}, when $F$ single-valued continuous. This condition is is also sufficient for the existence in case of a H-usc set-valued perturbation $F$ (see \cite[\S 4.5]{bothe-hab} and \cite{Pavel1}); see also \cite[Chap. 9]{CNV} for a detailed discussion of different tangency issues. Our study of $(K)$ along with $(F_4)$ is motivated by by Proposition \ref{uzasad0}, the second part of Proposition \ref{uzasad}  and Remark \ref{uzasad1}.
\end{rem}
\indent Let us now return to $\le\ast\pr$ and make the following assumptions:

$(\f_1)$ $\f:[0,1]\times \Omega \times K \m \R^N$  is usc with convex compact values;

$(\f_2)$ $\;\sup_{v\in \f \le t,x,y \pr} |v| \leq  \alpha \le x \pr + c |y|$, for some $\alpha \in L^2 \le \Omega \pr $, $c>0$ and for all $t\in [0,1]$, $x\in \Omega$ and $y \in K$.

$(\f_3)$ $\f$ is {\em tangent} to $K$, i.e. $\; \f \le t,x,y \pr \cap T_K\le y \pr \neq \emptyset$ for $t\in [0,1]$, $x\in \Omega$ and  $y\in K$ (\footnote{See \eqref{def tang} with $K$ replacing $\K$.})

\begin{rem}\label{uzasad1}
In order to understand the physical meaning of $(\f_3)$ consider an important special case when $K=\R^N_+$, i.e., $u=(u_1,...,u_N)\in K$ if and only if $u_i\geq 0$ for $i=1,...,M$. Interpreting $\le\ast\pr$ as the reaction-diffusion problem describing the dynamics of concentration  $u_1(x),...,u_N(x)$, $x\in\Omega$, of $N$ reacrants being subject to diffusion and reaction term, the usual
assumption of nonnegativity of $\f$ not realistic. Assumption $\f\geq 0$ implies that during chemical processes all substances are only produced, while, in fact, during reaction some reactants vanish or are transformed into another compounds. The realistic assumption is that if a reactant $i$ vanishes in some area (i.e., $u_i=0$), its amount in this area cannot decrease. This observation leads immediately to tangency (observe that if $u\in K$ with $u_i=0$, then $T_K(u)=\{v\in\R^n\mid v_i\geq 0\}$) $\f(u)\cap T_K(u)\neq\emptyset$ meaning exactly that $u_i$ can only increase.
\end{rem}

Put $\E:=L^2(\Omega,\R^N)$, $D \le A \pr : = H_0^1 \le \Omega, \R^N \pr \cap H^2 \le \Omega ,\R^N \pr $ and define $A : D \le A \pr \to \E$ by
\begin{equation}\label{opp}
A u := \le \Delta u_1 , \ldots , \Delta u_N \pr
\end{equation}
for $u=(u_1,...,u_N)\in D \le A \pr $, where $\Delta u_i$ denotes the usual Laplacian of a function $u_i:\Omega\to\R$. In view of \cite[Theorem 7.2.5]{pazy} $A$ generates an analytic and resolvent compact  semigroup of contractions $\left \{ S \le t\pr \right \}_{t \geq 0}$, i.e.,  $M=1$ and $\omega<0$ in Remark \ref{slabo-usc} (a).\\
\indent  Let
\begin{equation}\label{sett}
\K:= \left \{ v \in \E\mid v \le x \pr \in K \text{ for a.e. } x\in \Omega  \right \}.
\end{equation}
It is immediate to see that $\K$ is closed  convex. In order to get $(K)$, i.e., to show that  $J_h(\K)\subset\K$ when $h>0$ and $h\omega<1$ observe that $$K=\bigcap_{y\in K}(y+T_K(y))$$ and, hence, it is sufficient to consider the case when $K=y_0+C$, where $y_0\in\R^N$ and $C\subset\R^N$ is a closed convex cone. Then $\K=u_0+{\mathcal C}$ where $u_0(x)\equiv y_0$ on $\Omega$ and ${\mathcal C}:=\{u\in\E\mid u(x)\in C\;\hbox{a.e. on}\;\Omega\}$. Since $J_h(u_0)=u_0$, it is sufficient to show the invariance of $\mathcal C$. \\
\indent Take $v\in {\mathcal C}$. Since $C^\infty$ functions in $\mathcal C$ are dense in $\mathcal C$ we may assume that $v$ is $C^\infty$. Let $u=J_h(v)$; by the classical regularity theory $u\in C^\infty(\cl\Omega)$ and $u\restr{\partial\Omega}=0$. Let
$$p\in C\,^\ast:=\{p\in\R^N\mid \langle y,p\rangle\geq 0\;\hbox{for all}\;y\in C\}.$$ Then $v_p:=\langle p,v(\cdot)\rangle$, $u_p:=\langle p,u(\cdot)\rangle$ are $C^\infty$, $v_p\geq 0$ on $\Omega$ and $\langle p,Au(\cdot)\rangle=\Delta u_p$. Let $x_0\in\cl\Omega$ be such that $u_p(x)\geq u_p(x_0)$ for all $x\in\cl\Omega$. If $x_0\in\partial\Omega$, then $u_p(x)\geq 0$ on $\Omega$. If $x_0\in\Omega$, then the second derivative $D^2 u_p(x_0)$ is nonnegative; this implies that $\Delta u_p(x_0)\geq 0$. Hence
$$u_p(x_0)=\langle p, (u-hAu)(x_0)\rangle+h\langle  p,Au(x_0)\rangle= v_p(x_0)+h\Delta u_p(x_0)$$ and again $u_p\geq 0$ on $\Omega$. Since $p$ was arbitrary, we gather that $u(x)\in C$ on $\Omega$, i.e., $u\in {\mathcal C}$.\\
\indent We have shown
\begin{prop}\label{uzasad0} If $A$ is given by \eqref{opp} and $\K$ by \eqref{sett}, then assumptions $(A)$ and $(K)$ are satisfied.\hfill $\square$\end{prop}

Let $F:[0,1] \times \K \m \E$ be the {\em Nemytskii operator} associated with $\f$, i.e.
\begin{equation}\label{operator_niemyckiego}
F\le t, u \pr : = \left \{ v \in \E\mid v \le x \pr \in \f \le t , x , u \le x \pr \pr \text { for a.e. } x \in \Omega \right \}
\end{equation}
for $t\in [0,1], u \in \K$.
The values of $F$ are clearly nonempty, but {\em not compact} in general.
\begin{prop}\label{uzasad} If  $\f $ satisfies conditions $(\f_1) $ -- $(\f_3)$, then assumption $(F_1) - (F_3)$ are satisfied. In fact $F$ is H-usc (with respect to both variables). Assumption $(\f_3)$ implies $(F_4)$. Moreover any (mild) solution to \eqref{inkluzja-rozn} is a (strong) solutions to $\le\ast\pr$.
\end{prop}
\begin{proof} It is easy to see $(F_1)$ and $(F_3)$. Suppose $F$ is not H-usc, i.e., there are $\e _0>0$, sequences $\le t_n , u_n \pr \to \le t_0, u_0 \pr $ in $[0,1] \times\E$ and $v_n \in F \le t_n , u_n \pr $ such that
\begin{equation}\label{w-k_do_sprzecznosci}
v_n \notin F \le t_0 , u_0 \pr + B_\E(0,\e_0),\;\;n\geq 1.
\end{equation}
Up to a subsequence $\le u_n \pr _{n\geq 1}$ converges a.e. on $\Omega$ to $u_0$ and there is $h\in L^2 \le \Omega,\R \pr $ such that $|u_n \le x \pr | \leq h \le x \pr $ for a.e. $x\in \Omega$ and every $n\geq 0$.
By $\le \f _3 \pr $
\[
|v_n \le x \pr| \leq \alpha \le x \pr + c|u_n \le x \pr |  \leq \alpha \le x \pr + c h \le x \pr \text{ for }n\geq 0\text{ and }a.e.\; x\in \Omega.
\]
There is $\eta >0$ such that for $A \subset \Omega$ with Lebesgue measure $\mu \le A \pr <\eta$
\begin{equation}\label{wzor-absolutna-ciaglosc-calki}
\int _A  4\le \alpha\le x \pr + c h\le x \pr \pr ^2\mathrm{d}x  < \e_0^2 / 2.
\end{equation}
For each $n\geq 0$, $H_n:=\f \le t_n, \cdot, u_n \le \cdot \pr \pr :\Omega \m \R^N$ is measurable and
if $w:\Omega \to \R^N$ is a measurable selection of $H_n$, then $w \in\E$ since, in view of $\le\f_3\pr$, \begin{equation}\label{w-k_wzrostu}
|w\le x \pr | \leq \alpha \le x \pr + c h \le x \pr  \quad \text{for a.e. }x\in \Omega.
\end{equation}
\indent By the Egorov and Lusin theorems (see \cite[Th. 1]{averna} for a multivalued version of the Lusin theorem) there is a compact  $\Omega_\eta \subset \Omega$ such that $\mu \le \Omega \setminus \Omega_\eta \pr < \eta$, $u_n \to u_0$ uniformly on $\Omega_\eta$,
the restriction $u_0\restr {\Omega_\eta}:\Omega_\eta \to \R^N$ is continuous and  $H_0\restr {\Omega_\eta} : \Omega_\eta \m \R^N$ is H-lsc.\\
\indent Let $\delta := \e_0 /\sqrt{2\mu \le \Omega \pr }$. We will show that there is $n_0$ such that if $n\geq n_0 $ and $x\in \Omega_\eta$, then
\[
H_n \le x \pr \subset H_0 \le x \pr + B_{\R^N} \le 0, \delta \pr .
\]
Suppose to the contrary that  there is a subsequence $\le n_j \pr _{j\geq 1}$ and a sequence $\le x_j \pr _{j\geq 1 }$ in $\Omega_\eta$ such that
\begin{equation}\label{w-k_H_0}
H_{n_j} \le x_j \pr \not\subset H_0 \le x_j \pr + B_{\R^N} \le 0 ,\delta \pr .
\end{equation}
We can assume that $x_j \to x_0 \in \Omega_\eta$, since $\Omega_\eta$ is compact. The continuity of $u_0\restr {\Omega _\eta}$ and the uniform convergence $u_n \to u_0$ on $\Omega_\eta$ imply that $u_{n_j}(x_j)\to u_0(x_0)$ and thus $\le t_{n_j}, x_j, u_{n_j}\le x_j \pr \pr \to \le t_0, x_0, u_0 \le x_0 \pr\pr$ as $j\to \infty$. The upper semicontinuity of $\f$ together with the H-lower semicontinuity of $H_0$ on $\Omega _\eta $ show that
$H_{n_j}\le x_j \pr \subset H_0\le x_j \pr +B_{\R^N} \le 0 ,\delta\pr$
for sufficiently large $j$, which contradicts \eqref{w-k_H_0}.\\
\indent Let us fix $n\geq n_0$. For a.e. $x\in \Omega_\eta$ we have
\begin{equation}\label{w-k_selekcja_v_n}
v_n\le x\pr \in H_n \le x \pr \subset H_0\le x \pr + B_{\R^N} \le 0, \delta \pr.
\end{equation}
Observe that the map $\Omega_\eta \ni x\mm B_{\R^N} \le v_n \le x \pr ,\delta \pr \cap H_0\le x\pr$ is measurable and has nonempty values for a.e. $x\in \Omega_\eta$.
By the Kuratowski--Ryll-Nardzewski theorem, there is a measurable selection $v: \Omega_\eta \to \R^N $, i.e. $v\le x \pr \in B_{\R^N} \le v_n \le x \pr ,\delta \pr \cap H_0\le x\pr $ for a.e. $x\in \Omega_\eta$.
Clearly $v\in L^2 \le \Omega_\eta , \R^N\pr$ and for a.e. $x\in \Omega_\eta$, $|v_n\le x\pr -v\le x \pr| <\delta$. Thus
\[
\int _{\Omega_\eta} |v_n\le x \pr - v\le x \pr|^2\mathrm{d}x < \delta^2 \mu \le \Omega_\eta \pr < \e_0^2 / 2.
\]
Take an arbitrary selection $w:\Omega \to \R^N$ of $H_0$, i.e. $w\le x\pr \in H_0 \le x\pr $ for a.e. $x\in \Omega$.
Let $\chi$ be the indicator of $\Omega_\eta$.
Notice that $\chi v +(1-\chi)w:\Omega \to \R^N$ is a square-integrable selection of $H_0$ (we identify $v:\Omega_\eta \to \R^N$ with the function $v:\Omega \to \R^N$ putting $v\equiv 0$ on $\Omega\setminus \Omega_\eta$).
By \eqref{w-k_wzrostu}
\[
|v_n\le x\pr - w \le x \pr | \leq |v_n\le x\pr | + |w\le x\pr |\leq 2\le \alpha\le x \pr  + c h \le x \pr \pr\quad\text{for a.e. }x\in \Omega\setminus \Omega_\eta.
\]
Recall that  $\mu \le \Omega \setminus \Omega _\eta\pr<\eta$, hence and by \eqref{wzor-absolutna-ciaglosc-calki}
\begin{multline*}
\| v_n - \chi v+ (1- \chi)w\|^2= \int _{\Omega_\eta} |v _n\le x\pr - v \le x \pr |^2 \mathrm{d}x + \int_{\Omega \setminus \Omega_\eta}|v_n \le x \pr - w \le x  \pr |^2\mathrm{d}x \\
< \e_0^2/2 + \int _{\Omega \setminus \Omega_\eta} 4\le \alpha \le x \pr+ h\le x \pr \pr ^2 \mathrm{d}x <\e_0^2.
\end{multline*}
Thus, contrary to \eqref{w-k_do_sprzecznosci}, $v_n \in F \le t_0, u_0 \pr + B_{L^2 \le \Omega , \R^N \pr} \le 0, \e_0 \pr$ for infinitely many $n\geq 1$.\\
\indent In order to check $\le F_4\pr$ fix  $t\in [0,1]$, $u\in \K$ and define $G,H:\Omega \m \R^N, $ by
\[
G\le x \pr := \f \le t, x, u \le x \pr \pr, \; H\le x \pr := T_K \circ u \le x\pr \qquad \text{ for } x\in \Omega.
\]
The map $T_K(\cdot):K\m\R^N$ is lsc, $G$ is measurable; hence $\Omega\ni x\m G(x)\cap H(x)\subset\R^N$ is measurable with nonempty values. By  the Kuratowski--Ryll-Nardzewski theorem, there is a measurable  $v:\Omega \to \R^N$ such that $v\le x \pr \in G \le x \pr \cap H \le x \pr$ for a.e. $x\in \Omega$. Clearly $v\in\E$ and $v\in T_\K(u)\cap F(t,u)$ since in view of \cite[Cor. 8.5.2]{aubin} $T_\K(u)=\{v\in\E\mid v(x)\in T_K(u(x))\;\hbox{a.e. in}\;\Omega\}.$\\
\indent The last part has been established in \cite{gab-krysz} in the unconstrained case; this proof follows immediately from a general result in \cite[Proposition III.2.5]{Show}. Here the same arguments apply. \end{proof}

\section{Existence and structure of solutions}

In this section we assume that conditions
$\le A\pr$, $\le K \pr$, $\le F_1\pr $ -- $\le F_4\pr$ hold and $\E$ is a {\em separable} Banach space. The compactness of  $\{S(t)\}_{t\geq 0}$ implies that:
\begin{lem}\label{uzwarcanie}
The operator $K_0 :L^1 \le [0,1] ,\E \pr \to \mathcal{C}\le [0,1], \E \pr$ defined by
\[
K_0 \le y \pr \le t \pr := \int_0^t S \le t-s \pr y \le s \pr \mathrm{d}s\quad\text{for }y\in L^1 \le [0,1] ,\E \pr,\; t\in [0,1],
\]
maps integrably bounded subsets of $L^1 \le [0,1],\E\pr $ into compact subsets of $\mathcal{C}\le [0,1], \E \pr$.\hfill $\square$
\end{lem}

We are going to show that the set of all (mild) solutions to \eqref{inkluzja-rozn} surviving in $\K$ is a compact $R_\delta$-subset of $\mathcal{C}\le [0,1],\E\pr$, i.e., can be represented as the intersection of a decreasing sequence of compact absolute retracts (see also e.g. \cite[p. 14]{gorn} for a detailed discussion of the class of $R_\delta$-sets). In an unconstrained case this is known (see e.g. \cite{gab-krysz} or \cite{CWZ}). Assumptions $(K)$ and $(F_4)$ imply the {\em viability}, but certainly do not prevent that some solution escaper from $\K$; hence it is not clear what is the structure of solutions that stay in $\K$. Apart from the presence of constraints in \eqref{inkluzja-rozn}, we deal with weakly compact convex valued and  not necessarily usc perturbations, while elswhere (see e.g. \cite{koz} or \cite{bader-krysz})  compact convex valued usc perturbations are studied.\\
\indent In the proof the following characterization will be used: \emph{If $X_0=\bigcap_{n=1}^\infty X_n$, where $X_n\neq\emptyset$ is closed contractible,  $X_n\supset X_{n+1}$ for all $n\geq 1$, and the Hausdorff measure of noncompactness $\beta(X_n)\to 0$, then $X_0$ is an $R_\delta$-set}.
\begin{thm}\label{tw.o-istnieniu}
For a fixed $x_0\in \K$, the set $X_0$ of solutions in $\K$ of \eqref{inkluzja-rozn} starting at $x_0$ is an $R_\delta$ subset of $\mathcal{C}\le [0,1], \K \pr$.
\end{thm}

\begin{proof}
\textbf{Step 1.} Take a sequence $\le \e_n \pr _{n\geq 1}$ in $(0,1)$ such that $\e_n \searrow 0$.
Since $\K \ni x \mm T_\K \le x \pr \subset \E$ is lsc we can apply Theorem \ref{nguyen}:
%(cf. \cite[Rem. 19.]{cw}):
for every $n\geq 1$, there is $f_n : [0,1]\times \K \to \E$ such that
 $f_n\le t, \cdot \pr $ is continuous for $t\in [0,1]$ and $f_n \le \cdot ,x \pr $ is measurable for every $x\in \K$; $f_n \le t,x \pr \in T_\K \le x \pr $ and $f_n \le t,x \pr \in F \le \{t\} \times B_\K \le x , \e_n \pr \pr + B_\E \le 0,\e_n \pr $ for all $t\in [0,1], x\in \K$.\\
\indent For each $k\geq 1$, by a version of the Scorza-Dragoni theorem (cf. \cite{kucia}), there is a closed subset $\overline{I}_k\subset [0,1]$ such that
$\mu \le [0,1]\setminus \overline{I}_k \pr\leq\min\{\e_n/2^{k-n+1};n=1,...,k\}$ and the restriction $f_k\restr{\overline{I}_k \times \K}: \overline{I}_n\times \K \to \E$ is continuous (with respect to both variables).
Let $I_n := \bigcap_{k\geq n} \overline{I}_k$, $n\geq 1$. The family $\{I_n\}$  increases, consists of compact sets and $f_n\restr{I_n \times \K}$ is continuous. Moreover
$\mu \le [0,1] \setminus I_n \pr\leq \e_n$ and $\mu \le \bigcup_{n\geq 1} I_n \pr =1$.\\
\indent Fix $n\geq 1$; clearly $(0,1)\setminus I_n=\bigcup _{k\geq 1} \le a_k, b_k \pr $. Define $\widehat{f}_n:[0,1]\times \K \to \E$ by
\[
\widehat{f}_n \le t, x \pr := \begin{cases}f_n\le t,x\pr \qquad\qquad\qquad\qquad\qquad\quad\, \text{ for } t\in I_n, x\in \K\\
\frac{b_k-t}{b_k-a_k}f_n \le a_k, x \pr + \frac{t-a_k}{b_k-a_k}f_n\le b_k,x\pr \quad\text{ for } t \in [a_k,b_k], x\in \K.
\end{cases}
\]
Obviously $\widehat{f}_n$ is continuous and ($\conv$ stands for the convex hull)
\begin{equation}\label{p1}\widehat{f}_n\le t,x\pr \in \mathrm{conv}F \le B_{[0,1]}\le t,\e_n \pr \times B_\K\le x,\e_n\pr \pr +B_\E \le 0, \e_n \pr,\end{equation} for $t\in [0,1]$, $x\in \K$, since $b_k - a_k <\e_n$ for $k\geq 1$. If $t\in I_n$, $x\in \K$, then
\begin{equation}\label{p2}\widehat{f}_n \le t,x\pr = f_n \le t,x \pr \in F\le \{t\}\times B_\K\le x, \e_n \pr \pr + B_\E \le 0,\e_n\pr.\end{equation}
\indent  For any $n\geq 1$ we find easily a continuous $\alpha_n:[0,1]\times \K \to (0,\infty)$ such that if $g:[0,1]\times \K \to \E$ satisfies
\begin{equation}\label{p3}
g\le t,x\pr \in \widehat{f}_n \le B_{[0,1]} \le t,\alpha_n \le t,x\pr \pr \times B_\K \le x , \alpha_n \le t,x \pr \pr \pr + B_\E \le 0,\alpha_n \le t,x\pr \pr,
\end{equation}
then $g\le t,x\pr \in \widehat{f}_n\le t,x\pr + B _\E \le 0, \e_n \pr$ on $[0,1]\times\K$. Applying Lemma \ref{selekcja-styczna} to $\widehat{f}_n$, we get a locally Lipschitz  $g=g_n:[0,1]\times \K\to\E$ such that
\begin{equation}\label{p31}g_n \le t,x\pr \in T_\K\le x\pr\end{equation} and \eqref{p3} holds. Hence
\begin{equation}\label{p32}g_n \le t,x \pr \in \widehat{f}_n \le t,x \pr +B_\E \le 0,\e_n\pr.\end{equation}
In view of \eqref{p31}, $(A)$ and \cite[Th. 7.2.]{bothe}) the problem
$\dot{u} \le t\pr = A u \le t\pr + g_n \le t, u\le t\pr\pr$, $u\le 0 \pr = x_0$,
admits a unique solution (mild) $\overline{u}_n\in \mathcal{C}\le [0,1], \K\pr$.\\ \indent For any $n\geq 1$ let $X_n$ be the set of mild solutions (in $\K$) of the problem
\[
\begin{cases}
\dot{u}\le t\pr \in A u \le t\pr + F_n \le t, u\le t\pr \pr,\\
u\le 0 \pr =x_0 \in \K,
\end{cases}
\]where $F_n:[0,1] \times \K \m \E$ is given by
\[
F_n \le t,x\pr :=
\begin{cases}
F \le \{t\}\times B_\K \le x, \e_n \pr \pr + B_\E \le 0 ,2\e_n \pr \qquad \qquad\qquad \quad \text{ for } t\in I_n,x\in \K,\\
\mathrm{conv} F \le B_{[0,1]}\le t, \e_n\pr \times B_\K \le x,\e_n \pr \pr + B_\E \le 0,2\e_n\pr \quad\;\text{ for } t\in [0,1] \setminus I_n, x\in \K.
\end{cases}
\]
By \eqref{p32} and \eqref{p1}, \eqref{p2},
$g_n\le t,x\pr \in F_n \le t,x\pr $ on $[0,1]\times \K$. Therefore $X_n\neq\emptyset$ since  $\overline{u}_n\in X_n$. Clearly
\begin{equation}\label{p33}X_0\subset\bigcap_{n=1}^\infty X_n.\end{equation}
\indent \textbf{Step 2.} We shall see that given a sequence $(u_n)$, where $u_n\in X_n$ for $n\geq 1$, then (up to a subsequence) $u_n\to u_0\in X_0$. To this end observe that there is $ w_n \in L^1 \le [0,1], \E \pr $ such that $w_n \le t\pr \in F_n \le t, u_n \le t\pr \pr$ for a.e. $ t\in [0,1]$ and $u_n \le t \pr = S \le t \pr x + K_0(w_n)$ for $t\in [0,1]$.
The Gronwall inequality and $(F_3)$ imply that $\sup_{n\geq 1}\|u_n\| \leq C$ for some $C\geq 0$. Thus $\left \{ w_n\right \}_{n\geq 1}$ is integrably bounded by $c\le 1 + C \pr$ and, by Lemma \ref{uzwarcanie},  $\left \{ u_n \right \}_{n\geq 1}$ is relatively compact, i.e., (up to a subsequence) $ u_n \to u_0 \in \mathcal{C}\le [0,1], \K \pr$.\\
\indent Observe now that the set  $\left \{\chi_nw_n\right \}_{n\geq 1}$, where $\chi_n$ stands for the indicator of $I_n$,  is integrably bounded. Take $t\in \bigcup_{n\geq 1}I_n$, i.e., $t\in I_n$ for $n\geq N$ for some $N$. For such $n$
\begin{equation}\label{p4}
\chi_n \le t \pr w_n \le t \pr  = w_n\le t \pr \in F \le \{t \}\times B_\K \le u_n \le t \pr , \e_n \pr \pr + B_\E \le 0, 2\e_n \pr,
\end{equation}
i.e.,  $w_n \le t \pr \in F \le t,v_n \pr +b_n$, for some $v_n\in \K, b_n \in \E$ with $\|u_n \le t \pr - v_n \| <\e_n, \|b_n \|<2\e_n$. Hence
\[
\|u_0 \le t\pr - v_n \| \leq \|u_0(t) - u_n(t)\| + \| u_n \le t \pr - v_n\| \to 0 \quad \text{as } n \to \infty.
\]
Observe that
\[
\left \{ w_n \le t\pr \chi_n \le t\pr \right \}_{n\geq 1}\subset F \le \{t\} \times \left \{ v_n \right \}_{n\geq 1} \pr + \left \{b_n \right \}_{n\geq 1} \cup \left \{ 0 \right \},
\]
where $F \le \{t\}\times \left \{ v_n \right \}_{n\geq N}\pr $ is relatively weakly compact in view of Remark \ref{slabo-usc} (c). By the Diestel weak compactness criterion \cite[Cor. 2.6]{diest},  $\left \{\chi_n w_n\right \}_{n\geq 1}$ is relatively weakly compact in $L^1 \le [0,1] ,\E \pr $, i.e.,  up to a subsequence $\chi_n w_n\rightharpoonup w_0\in L^1 \le [0,1] ,\E\pr $ (weakly) and, hence,
 $\le K_0 \le w_n \chi_n\pr\pr _{n\geq 1} \rightharpoonup K_0 \le w_0\pr$ in $\mathcal{C} \le [0,1], \E\pr$.\\
\indent On the other hand $\left \| K_0 \le (1-\chi_n)w_n\pr \right \| \leq Rc\le 1+ M \pr\mu \le [0,1] \setminus I_n \pr \to 0$ as $n\to\infty$. Therefore
$$u_n-S(\cdot)x=K_0(\chi_nw_n)+K_0((1-\chi_n)w_n)\rightharpoonup K_0(w_0).$$
This shows $u_0 \le t \pr = S \le t\pr x+ \int _0^t S \le t-s \pr w_0 \le s \pr \mathrm{d}s$ for $t\in [0,1]$. In view of \eqref{p4} and the `convergence theorem'  \cite[Th. 3.2.6]{ekeland}, $w_0 \le t\pr\in F \le t,u_0 \le t \pr \pr $ for a.e. $t\in [0,1]$, i.e., $u_0\in X_0$.\\
\indent The assertion we have just proved together with \eqref{p33} implies that $X_0$ is compact, $\sup_{v\in X_n} d \le v ; X_0\pr\to 0$ and, hence, the measure of noncompactness $\beta(\cl X_n)\to 0$ and $X_0=\bigcap_{n=1}^\infty\cl X_n$.

\indent \textbf{Step 3.} Now we shall show that $\cl X_n$ is contractible. To see this fix $n\geq 1$ and recall  the above constructed locally Lipschitz   $g_n :[0,1] \times\K \to \E$ being tangent to $\K$ and having sublinear growth. Take $z\in [0,1]$ and $y\in \K$. The problem
\[
\begin{cases}
\dot{u}\le t\pr = A u\le t\pr + g_n \le t, u\le t\pr \pr , \\
u\le z \pr = y,
\end{cases}
\]
admits a unique solution $v\le \cdot ; z,y\pr : [z,1] \to \K$.
The strong continuity of $\left \{ S \le t\pr \right \}_{t\geq 0}$ along with local lipschitzeanity of $g_n$ imply that $v\le \cdot; z,y \pr $ depends continuously on $z\in [0,1]$ and $y\in \K$.
Precisely, given $\e>0$, $z_0\in [0,1]$ and $y_0 \in \K$ there is $\delta >0$ such that $\|v\le t; z_0 , y_0 \pr - v \le t; z ,y \pr \| < \e$ for all $t\in \left [ \max \{z_0,z \} ,1 \right ]$, if $|z-z_0| <\delta, \|y-y_0\| <\delta$.\\
\indent Let us consider the homotopy $h:\cl X_n \times [0,1] \to \mathcal{C}\le [0,1] ,\K \pr$ given by
\[
\left [h\le u, z\pr \right ] \le s \pr :=
\begin{cases}
u\le s \pr \qquad \qquad \qquad\;\text{ for } s\in [0,z];\\
v\le s ; z, u \le z \pr \pr \qquad\quad \text{ for } s\in [z,1]
\end{cases}
\]
where $u\in \mathrm{cl}X_n , z\in [0,1]$. It is easy to see that $h$ is well-defined, continuous (comp. \cite[Th. 5.1]{bothe-hab}) and $ h \le X_n \times [0,1] \pr\subset X_n$ since $g_n$ is the selection of $F_n$; thus
$h \le \mathrm{cl}X_n \times [0,1] \pr \subset \mathrm{cl}X_n $.
Furthermore $h\le  \cdot , 0 \pr = v \le \cdot ; 0,x \pr$ and $h \le \cdot ,1 \pr = \mathrm{id}_{\mathrm{cl}X_n}$ proving the contractibility of $\mathrm{cl}X_n$.
\end{proof}

\subsection{$c$ - admissible maps}  Recall (see \cite{Lacher} and \cite{Hyman}) that a compact metric space space $S$ is {\em cell-like} if  it can be represented as the intersection of a decreasing sequence of compact contractible spaces. The following conditions are equivalent (see e.g. \cite{Hyman}): $S$ is cell-like; $S$ has the shape of a point; $S$ is an $R_\delta$-set; $S$ has the {\em $UV^\infty$-property}, i.e., if $S$ is embedded into an ANR, then it is contractible in any of its neighborhoods.\\
\indent Let $X, Y$ be metric spaces; an usc map $\f\colon X \m Y$ is \emph{cell-like} if  $\f\le x \pr $, $x\in X$, is  cell-like. A map $\f:X \m Y$ is $c$-\emph{admissible} if there is a metric space $Z$, a cell-like map $\psi:X\m Z$ and a continuous $f:Z\to Y$ such that $\f=f\circ\psi$. Equivalently (see \cite[Section 3]{gab-krysz-ind}) $\f:X\m Y$ is $c$-admissible if it is {\em represented} by a {\em $c$-admissible pair} $\le p,q\pr$, i.e., $\f \le x \pr = q \le p^{-1} \le x \pr \pr $ for $x\in X$, where $X\xleftarrow{p} \Gamma \xrightarrow{q} Y$, $\Gamma$ is a metric space, $p,q$ are continuous and $p$ is a proper surjection with cell-like  $p^{-1}\le x\pr$, $x\in X$. Properties of a $c$-admissible $\f$ strongly depend on a decomposition $\f=f\circ\psi$ or a pair $(p,q)$ representing it. When studying $c$-admissible maps one has to take into account representing pairs (for a detailed discussion of $c$-admissible maps, related topics and some references -- see \cite{gab-krysz-ind}). In particular: if $\f:X\m Y$ is cell-like, then the {\em canonical} pair $(p_\f,q_\f)$, where the {\em graph} $\Gr(\f):=\{(x,y)\in X\times Y\mid y\in\f(x)\}$, $p_\f:\Gr(\f)\to X$ and $q_\f:\Gr(\f)\to Y$ are projections, is $c$-admissible and represents $\f$. If $X \subset \E$, then a $c$-admissible pair $\le p,q \pr$ is \emph{compact} if $\cl q\le p^{-1} \le B \pr \pr $ is compact for any bounded $B\subset X$; $\f :X \m Y$ is  {\em compact} if represented by a compact $c$-admissible pair.\\
\indent After \cite[Definition 3.5]{gab-krysz-ind} we say that
$c$-admissible pairs  $X\xleftarrow{p_k} \Gamma_k \xrightarrow{q_k} Y$, $k=0,1$, (and set-valued maps represented by them) are {\em $c$-homotopic}
(written $(p_0,q_0)\simeq (p_1,q_1)$) if there is a $c$-admissible
pair $X\times [0,1]\xleftarrow{p} \Gamma \xrightarrow{q} Y$
and continuous maps $j_k:\Gamma_k\to\Gamma$, $k=0,1$, such that the following
diagram
$$\begin{gathered}\label{hom}\xymatrix@R-9pt{X\ar[d]_-{i_0}&
\Gamma_0\ar[l]_-{p_0}\ar[d]_-{j_0}\ar[dr]^-{q_0}&\\
X\times [0,1]&\Gamma\ar[l]_-p\ar[r]^-q&Y,\\
X\ar[u]^-{i_1}&\Gamma_1\ar[l]^-{p_1}\ar[u]^-{j_1}\ar[ur]_-{q_1}&}\end{gathered}$$
where $i_k(x):=(x,k)$ for $x\in X$ and $k=0,1$, is commutative.
The pair $(p,q)$ is called a {\em $c$-homotopy} joining
$(p_0,q_0)$ to $(p_1,q_1)$.

\bigskip

\indent Let  $\Sigma : \K \m \mathcal{C}\le [0,1],\K\pr$ assign to $x\in \K$ the set of all solutions to \eqref{inkluzja-rozn} starting at $x$.
\begin{lem}\label{lem-gorna-polciaglosc}
$\Sigma$ is a cell-like map and maps bounded sets onto bounded ones.
\end{lem}
\begin{proof} The second assertion follows from the Gronwall inequality and $\le F_5\pr$. In view of Theorem \ref{tw.o-istnieniu} we need to show that $\Sigma$ is usc. Let $x_n \to x \in \K$ and $u_n \in \Sigma \le x_n \pr $ for $n\geq 1$.
Then $u_n = S \le \cdot \pr x_n + K_0 \le w_n \pr $ for some $w_n \in L^1 \le [0,1], \E \pr $ such that $w_n\le t \pr \in F \le t, u_n \le t \pr \pr $ for a.e. $t\in [0,1]$.
The condition $\le F_3 \pr$ and the Gronwall inequality imply that $\left \{ u_n \right \}_{n\geq 1}$ is bounded, so $\left \{ w_n \right \}_{n\geq 1}$ is integrably bounded. As above (up to a subsequence) $u_n - S\le \cdot \pr x_n \to u - S \le \cdot \pr x$ in $\mathcal{C}\le [0,1],\K\pr$.
Thus, again up to a subsequence $w_n\rightharpoonup w\in L^1 \le [0,1], \E \pr $ and $w\le t\pr \in F \le t, u \le t \pr \pr $ for a.e. $t\in [0,1]$.
As a result, $u = S \le \cdot\pr x + K_0 \le w \pr \in \Sigma \le x\pr $.
\end{proof}
In what follows $\Sigma$ will be identified it with its canonical pair
\begin{equation}\label{op_rozw}
\K \xleftarrow{p}\Gamma \xrightarrow{q} \mathcal{C}\le [0,1] , \K \pr,\;\; \Sigma(x)=q_\Sigma(p_\Sigma^{-1}(x)),\;x\in \K,
\end{equation}
where $\Gamma := \left \{ \le x,u \pr \in \K \times \mathcal{C} \le [0,1], \K\pr \mid u \in \Sigma \le x \pr \right \}$ is the graph of $\Sigma$, $p_\Sigma$ and $q_\Sigma$ are the projections onto $\K$ and into $\mathcal{C}\le [0,1], \K\pr$, respectively.

For a fixed $t\in [0,1]$, the evaluation $e_t : \mathcal{C}\le [0,1],\K \pr \to \K$, $e_t \le u \pr: = u \le t\pr $ for $u\in \mathcal{C}\le [0,1], \K \pr$ is defined and continuous. With \eqref{inkluzja-rozn} we associate  the {\em Poincar\'{e} $t$-operator} $\Sigma_t :\K \m \K$,
\begin{equation}\label{op_poi}\Sigma_t := e_t \circ \Sigma,\;\;\hbox{i.e.},\;\Sigma _t \le x \pr = \left \{u \le t \pr \mid u \in \Sigma \le x\pr  \right \}.\end{equation}
Therefore $\Sigma_t$ is  c-admissible (cf. \cite[Rem. 3.4. (2)]{gab-krysz-ind}); it is represented by the c-admissible pair
\begin{equation}\label{c-adm}
\K\xleftarrow{p_t} \Gamma \xrightarrow{q_t} \K, \;\hbox{where}\;\; p_t:=p_\Sigma,\;q_t:=e_t\circ q_\Sigma.
\end{equation}

\begin{rem}\label{uwaga-c-admisible-solution-mapping}
(1) The mapping $\K\times [0,1] \ni \le t,x \pr \mm \Sigma_t \le x\pr \subset \mathcal{C}\le [0,1], \K\pr$ is c-admissible. Is it represented by the pair
\[
\K\times [0,1] \xleftarrow{p} \Gamma \times [0,1] \xrightarrow{q} \K
\]
where $p:=p_\Sigma \times \mathrm{id}_{[0,1]}, q\le \gamma ,t \pr := e_t \circ q_\Sigma \le \gamma \pr $ for $t\in [0,1], \gamma \in \Gamma $.\\
\indent (2) For any numbers $0<a \leq b \leq 1$, the restriction $[a,b]\times \K \ni \le t,x \pr \mm \Sigma_t \le x \pr \subset \K$ is completely continuous, what is a consequence of the compactness of the semigroup $\left \{ S \le t \pr \right \}_{t\geq 0}$.
It is worth to emphasize that, in particular, $\Sigma_t $ and the  pair $\le p_t , q_t \pr$ are completely continuous.
\end{rem}
\begin{rem}\label{uwaga-parametryczna-wersja} A parameterized version of the above results will also be useful.
Let $Z$ be a compact metric space and let $F:Z\times [0,1] \times \K \m \E$ be (product) measurable, $F\le \cdot, t , \cdot \pr$, $t\in [0,1]$ be H-usc and $F(z,\cdot,\cdot)$, $z\in Z$,  satisfy assumptions $\le F_1 \pr$ - $ \le F_4\pr$. Then all above results remain true, in particular: the solution map $\Sigma : Z\times \K \m \mathcal{C}\le [0,1] , \K\pr$ is usc with $R_\delta$-values.
\end{rem}
\subsection{Fixed point index for $c$ - admissible maps}\label{inde}
Given an open bounded $V\subset\E$, a compact $c$-admissible pair $\cl V\xleftarrow{p}\Gamma \xrightarrow{q}\E $ representing it, such that $x\notin q \le p ^{-1} \le x \pr \pr $ for $x\in\partial V$, the fixed point index $\mathrm{Ind}\le \le p,q\pr , V\pr$ is well-defined  (cf. \cite[Th. 4.5]{gab-krysz-ind}). This index has the usual properties such as: the existence, the localization, the additivity and the homotopy invariance (see \cite{gab-krysz-ind}).\\
\indent It is easy to get a generalization of the above mentioned fixed point index to a constrained case in a standard way. Let $\K\subset\E$ be convex closed and let $U\subset \K$ be (relatively) open and bounded. Let $r:\E \to \K$ be an arbitrary retraction and $j:\K \hookrightarrow \E$ be the inclusion. Given a $c$-admissible compact pair $\cl_\K U\xleftarrow{p}\Gamma \xrightarrow{q}\K$ such that $x\notin q \le p^{-1} \le x \pr \pr $ for $x\in \partial_\K U$ ($\cl_\K U$ and $\partial_\K U$ denote the closure and the boundary of $U$ in $\K$), we let $V:= r^{-1} \le U \pr\cap B$, where $B$ is open bounded and $B\supset U$, $\bar\Gamma:= \left \{ \le x, \gamma \pr \in \mathrm{cl}U \times \Gamma \mid r \le x \pr =p\le \gamma \pr \right \}$,  $\bar p:\bar\Gamma \to \cl U$ and $\bar q :\bar \Gamma\to\E$ by $\bar p\le x,\gamma \pr := x$ and $\bar q\le x,\gamma\pr:= q\le \gamma\pr $ for $\le x,\gamma\pr \in \bar\Gamma $. Note that $q_r \circ p_r^{-1} =  j\circ q\circ p^{-1} \circ r_U:\mathrm{cl}U_r\m E$,
the pair $\le p_r,q_r\pr$ is  compact and $c$-admissible and $x\not\in\bar q(\bar p^{-1}(x))$ for $x\in\partial V$.
Thus we are in a position to define the {\em constrained fixed point index} by
\[
\mathrm{Ind}_\K \le \le p,q\pr , U\pr := \mathrm{Ind}\le \le p_r, q_r \pr , \mathrm{cl}U_r\pr.
\]
It is easy to see that this definition is correct, i.e, it does not depend on the choice of $r$; furthermore $\mathrm{Ind}_\K$ has the same properties as $\mathrm{Ind}$ does.
\begin{rem}\label{uwaga-stopien-party-jednowartosciowej} (i) In particular, if two $c$-admissible pairs $\cl_\K U\xleftarrow{p_j}\Gamma_j \xrightarrow{q_j}\K$, $j=0,1$, are $c$-homotopic and the $c$-homotopy $\cl_\K U\times [0,1]\xleftarrow{p}\Gamma \xrightarrow{q}\K$ is compact and such that $x\not\in q(p^{-1}(x,t))$ for $x\in\partial_\K U$, $t\in [0,1]$, then $\mathrm{Ind}_\K((p_j,q_j),U)$, $j=0,1$, are defined and equal.\\
\indent (ii) If a compact $c$-admissible pair $\le p,q \pr $ represents a single-valued $f:\mathrm{cl}\,U\to \K$ and $x\neq f\le x \pr $ for $x\in \partial U$, then it can be proved $\mathrm{Ind}_\K \le \le p,q \pr , U \pr = \mathrm{Ind}_\K \le f , U \pr$, where $\mathrm{Ind}_\K(f,U)$ stands for the fixed point index as defined in \cite[\S 12]{dg}. In particular $f$ is represented by the pair $\cl_\K U \xleftarrow{\mathrm{id}}\cl_\K U\xrightarrow{f}\K$
\end{rem}

\section{The degree of the right hand side}
We will construct a homotopy invariant (the so-called {\em constrained topological degree}) responsible for the existence of zeros of maps of the form $A+G$, where:

$(G_1)$ \parbox[t]{0.78\textwidth}{ $G:\K\m \E$ is H-usc, has convex weakly compact values, maps bounded sets onto bounded ones and $G\le x\pr \cap T_\K\le x\pr\neq \emptyset$  for every $x\in \K$, i.e., $G$ is tangent to $\K$;}

$(G_2)$ \parbox[t]{0.78\textwidth}{$\K\subset\E$ is convex closed; $A:D \le A \pr \to \E$ satisfies  $\le A\pr$ and $\le K\pr$.}

\noindent Let $U\subset \K$ be  bounded and relatively open in $\K$.  We assume that
\begin{equation}\label{zal-brak-pkt-st-na-brzegu}
0\not\in A x + G \le x \pr \quad\text{for } x\in D\le A \pr \cap \partial U;
\end{equation}
here $\partial U = \partial _\K U$ stands for the boundary of $U$ in $\K$.
\begin{lem}\label{lem-do-konstr-stopnia}
There is $\alpha_0 >0$ such that if $0<\alpha\leq\alpha_0$, then
\[
0\notin A x +G \le B_\K \le x, \alpha \pr \pr + B_\E \le 0, \alpha \pr  \quad\text{for } x\in D \le A \pr \cap \partial U.
\]
\end{lem}
\begin{proof}
Suppose to the contrary that for $n\geq 1$ there is $x_n \in D \le A \pr \cap \partial U$, $y_n\in G(\bar x_n)$, where $\|x_n-\bar x_n\|<1/n$ and $\xi_n\in\E$ with $\|\xi_n\|<1/n$  such that
\[
0 =A x_n + y_n + \xi_n \quad \Longleftrightarrow \quad x_n = J_h \le x_n +h \le y_n + \xi_n \pr \pr
\]
for fixed $h>0, h\omega <1$.
Clearly $\left \{ y_n \right \}_{n\geq 1}$ is bounded since so is $\{x_n\}$.
The compactness of $J_h$ implies that $\left \{ x_n \right \}_{n\geq 1}$ is relatively compact; thus, up to a subsequence,  $x_n \to x_0\in \partial U$ and $\bar x_n\to x_0$. Remark \ref{slabo-usc} (c) and the Krein-\v{S}mulian theorem imply that $\left \{ y_n\right \}_{n\geq 1}$ is relatively weakly compact.
Thus, up to a subsequence $y_n \rightharpoonup y_0$. This (see again  Remark \ref{slabo-usc} (c)) implies that $y_0\in G(x_0)$. Moreover $x_n=J_h \le x_n + h \le y_n + \xi_n \pr \pr \to J_h\le x_0 + h y_0\pr $, since $J_h$ is compact.
Hence $x_0 = J_h \le x_0 + hy_0\pr $, $x_0\in D\le A \pr$ and $0=A x_0 + y_0$: a contradiction.  \end{proof}
\begin{lem}\label{lem-drugie-podejscie-stycznosc}
If a continuous map $g:\K\to \E$ is tangent to $\K$, then for every $x\in \K$ we have
\[
\lim _{h\to 0^+, y\to x, y\in \K} \frac{d\le J_h \le y+hg\le y\pr \pr;\K\pr}{h}=0.
\]
\end{lem}
\begin{proof}
Take $x\in \K$ and $\e>0$.
The continuity and the tangency of $g$ together with \cite[Prop. 4.2.1]{aubin} imply
\[
\lim _{h\to 0^+, y\to x, y\in \K} \frac{d\le  y+hg\le y\pr;\K\pr}{h}=0\quad\text{for }x\in \K.
\]
Hence (see Remark \ref{slabo-usc} (a)), there is $\delta>0$ such that if $\|y-x\|<\delta,\, 0<h<\delta$.
\[
d\le y + h g \le y \pr ; \K\pr <\frac{\e}{2M}h \quad\text{and}\quad \frac{M}{1-hw}< 2M.
\]
Choose $k\in \K$ with $\|y+hg\le y\pr - k\| < \e h\le 2M \pr ^{-1}$.
For $e:= \le k-y-hg\le y\pr\pr /h$,  $\|e\| <\e/2M$ and $y+ h \le g\le y\pr + e \pr = k \in \K$. Assumption $(K)$ implies $J_h \le y+ h \le g\le y\pr + e \pr\pr\in \K$. Thus
\[
d\le J_h \le y + h g\le y\pr \pr;\K \pr \leq \| J_h \le y + h g \le y\pr \pr -  J_h \le y+ h \le g\le y\pr + e \pr\pr\| \leq h\| J_h \| \|e\| <h\e
\]
if $\|y-x\|<\delta,\, 0<h<\delta$.
\end{proof}

Let $r:\E\to \K$ be a retraction, such that  $\|x-r(x)\|\leq 2d(x;\K)$ for $x\in\E$; such retractions exist.
\begin{lem}\label{lem-do-konstrukcji-stopnia-2}
Assume that $g:
K\to\E$ is continuous and tangent $\alpha$-approximation of $G$, i.e., $g\le x \pr \in G \le B_\K \le x ,\alpha \pr \pr + B_\E \le 0, \alpha \pr $ for $x\in \K$, where $0<\alpha\leq\alpha_0$ (see Lemma \ref{lem-do-konstr-stopnia}).
Then there is $h_0 >0$, $h_0 \omega <1$ such that for $h\in (0,h_0 ]$
\[
x\neq r\circ J_h \le x + h g \le x \pr \pr \quad \text{for }x\in \partial U.
\]
\end{lem}
\begin{proof} If not, then for each $n\geq 1$ there is $x_n\in \partial U$ such that $x_n = r \circ J_{h_n} \le x_n + h_n g\le x_n \pr \pr$, where $0<h_n<1/n$ and $h_n\omega<1$.
Denoting $u_n := J_{h_n} \le x_n + h_n g\le x_n \pr \pr \in D\le A\pr$ we have
$h_n^{-1}d(u_n;\K)\to 0$ in view of Lemma Lemma \ref{lem-drugie-podejscie-stycznosc} and
$$u_n-r(u_n)=u_n-x_n=h_n(Au_n+g(x_n)).$$
Hence
\begin{equation}\label{wzor-wewnatrz-lematu-do-konstr-stopnia}
 \left \| A u_n + g \le x_n \pr \right \| =\frac{1}{h_n}\| u_n - r\le u_n \pr \|\leq \frac{2}{h_n}  d(u_n;\K)\to 0.
\end{equation}
Thus $\left \{ A u_n \right \}_{n\geq 1}$ is bounded since so is $\left \{ g\le x_n \pr \right \}_{n\geq 1}$.
Note that $\|u_n \| \leq \|J_{h_n}\|\|x_n + h_n g \le x _n \pr \|\leq R$ for some $R>0$. Fix  $h>0$, $h\omega <1$. The compactness of $J_h$ and $u_n = J_h \le u_n - h Au_n \pr $ implies that, up to a subsequence, $u_n \to x_0 \in \E$. Since $d(u_n;\K)\to 0$, we infer that
$x_0\in \K$ and $x_n  = r\le u_n \pr \to r\le x_0 \pr =x_0\in \partial U$.
In view of \eqref{wzor-wewnatrz-lematu-do-konstr-stopnia} $A u_n \to - g \le x_0 \pr $ and since $A$ is closed we have $x_0 \in D \le A \pr$ and $Ax_0 = - g\le x_0 \pr$.
As a result $x_0 \in D \le A \pr \cap \partial U$ and $0 = Ax_0 + g \le x_0\pr$: a contradiction to Lemma \ref{lem-do-konstr-stopnia}.
\end{proof}

By Lemma \ref{selekcja-styczna} there is a locally Lipschitz  $g:\K\to \E$ tangent to $\K$ being an $\alpha$-approximation of $G$.
Let  $h \in (0,h_0]$ ($h_0$ is taken from Lemma \ref{lem-do-konstrukcji-stopnia-2}) and consider $f :\mathrm{cl}\,U\to \K$ defined by
\[
f \le x \pr := r\circ J_h \le x + h g\le x \pr \pr \quad\text{for } x\in \mathrm{cl}\,U.
\]
Obviously, $f$ is compact and by Lemma \ref{lem-do-konstrukcji-stopnia-2}, $x\neq f(x)$ for $x\in \partial U$. Thus, the fixed point index in ANRs $\mathrm{Ind}_\K\le f, U \pr $ is well-defined (see \cite[\S 12]{dg})

\begin{lem}\label{pop-def} The number $\mathrm{Ind}_\K\le f, U \pr $  does not depend on the choice of a sufficiently close approximation $g$, a retraction $r$ and sufficiently small $h>0$.\end{lem}
\begin{proof} Take two retraction $r_0,r_1:\E\to\K$ such that $\|x-r_i(x)\|\leq 2d(x;\K)$, $x\in\E$, $i=0,1$, and two locally Lipschitz $\alpha$-approximations $g_0, g_1:\K \to \E$ of $G$ tangent to $\K$, where $0<\alpha\leq\alpha_0$. If Repeating arguments form Lemmata \ref{lem-do-konstr-stopnia} and \ref{lem-do-konstrukcji-stopnia-2} we find a sufficiently small $\alpha>0$ and $h\leq h_0$ such that for any $t\in [0,1]$
$$x\neq r_t\circ
f_t(x):=J_h(x+hg_t(x))\;\;\hbox{on}\;\; \partial U,$$
where $r_t:=(1-t)r_0+tr_1$ and $g_t=(1-t)g_0+tg_1$. Thus $\partial U\times [0,1]\ni (x,t)\mapsto f_t(x)$ provides a (compact) homotopy joining $f_0$ to $f_1$ showing that $\mathrm{Ind}_\K\le f_0, U \pr=\mathrm{Ind}_\K\le f_1, U \pr $.  The independence of $\mathrm{Ind}_\K\le f, U \pr $ follows easily from the resolvent identity
$$J_b = J_a \left [ \frac{a}{b}I + \frac{b-a}{b}J_b\right],$$
being valid for any $a,b>0$ with $a\omega, b\omega <1$ and again the homotopy invariance of the fixed point index.\end{proof}
Thus, we are in a position to define the degree $\mathrm{deg}_\K$ by
\begin{equation}\label{definicja-stopnia}
\mathrm{deg}_\K \le A + G , U\pr := \lim_{h\to 0^+}\mathrm{Ind}_\K \le r\circ J_h \le I + hg \pr,U\pr
\end{equation}
where $g:\K\to \E$ is a tangent and sufficiently close locally Lipschitz approximation of $G$.
\begin{prop}
The degree $\mathrm{deg}_\K $ has the following basic properties:\\
(1) (Existence) If $\mathrm{deg}_\K \le A + G , U \pr \neq 0$, then there is $x\in D \le A \pr \cap  U$ such that $0\in A x + G\le x \pr $;\\
(2) (Additivity) If $U_1,U_2 \subset U$ are disjoint open in $\K$ and
$0\not\in Ax=G(x)$ for $x\in D(A)\cap [\cl U\setminus (U_1\cup U_2)]$,
then
\[
\mathrm{deg}_\K \le A +G , U \pr = \mathrm{deg}_\K \le A +G , U_1 \pr + \mathrm{deg}_\K \le A +G, U_2 \pr .
\]
(3) (Homotopy invariance) If $H:[0,1]\times \mathrm{cl}\,U \m \E $ is H-usc with convex weakly compact values, maps bounded sets onto bounded ones and is tangent to $\K$, i.e., $H(t,x)\cap T_\K(x)\neq\emptyset$, $t\in [0,1]$, $x\in\partial U$, and such that $0\notin A x + H \le t,x\pr $ for $t\in [0,1], x\in \partial U$, then
\[
\mathrm{deg}_\K \le A + H \le 0,\cdot \pr , U\pr = \mathrm{deg}_\K \le H \le 1, \cdot \pr, U \pr.
\]
\end{prop}
\begin{proof}
Suppose $0\not\in Ax+G(x)$ for $x\in\cl U\cap D(A)$. Arguing as in Lemmata \ref{lem-do-konstr-stopnia} and \ref{lem-do-konstrukcji-stopnia-2} we find $0<\alpha_1\leq\alpha_0$ and $0<h_1\leq h_0$ such that for any $0<\alpha\leq\alpha_1$ and any locally Lipschitz and tangent $\alpha$-approximation $g:\K\to\E$, $x\neq r\circ J_k(x+hg(x))$ for $x\in\cl U$, where $0<h\leq h_1$. This shows that $\mathrm{deg}_\K \le A + G , U\pr=0$.\\
\indent The remaining assertions are standard and left to the reader.
\end{proof}

\section{The Krasnosel'skii type formula}

In this section we will  prove the following counterpart of the classical Krasnoselskii formula by establishing a formula relating the constrained degree of the operator $A+F(0,\cdot)$ in the right-hand side of \eqref{inkluzja-rozn} and the fixed point index of the Poincar\'e operator $\Sigma_t$ (with sufficiently small $t>0$) associated to \eqref{inkluzja-rozn}; see \eqref{op_poi}, \eqref{c-adm}.
\begin{thm}\label{glowne}
Assume that operator $A:D\le A \pr \to \E$, where $\E$ is a separable Hilbert space, and $\K$ satisfy hypotheses $\le A \pr$, $\le K \pr $ and, additionally let $\| S \le t \pr \|\leq e^{\omega t} $ for some $\omega \in \R$ and all $t\geq 0$. Let $F :[0,1]\times \K \m \E$ satisfy conditions $\le F_1 \pr $, $(F_3)$ an d$(F_4)$  and, instead of $(F_2)$, we assume that
$$F:[0,1] \times \K \m \E\;\;\hbox{is H-usc}.\leqno (F)$$
Let $U \subset \K$ be bounded relatively open in $\K$ and
$0\notin Ax + F \le 0 , x \pr$ for $x\in \partial U \cap D \le A\pr$.
There is $t_0\in (0,1]$ such that if $t\in (0,t_0]$, then
$\mathrm{Ind}_\K \le  \le p_t ,q_t \pr , U \pr$ is well-defined and equal to $\mathrm{deg}_\K \le A + F \le 0, \cdot \pr , U \pr$.
\end{thm}
\noindent  Observe that $F(0,\cdot)$ satisfies $(G_1)$ and $(G_2)$; hence $\mathrm{deg}_\K \le A + F \le 0, \cdot \pr , U \pr$ is well-defined

The proof of Theorem \ref{glowne} will be presented in a series of steps and auxiliary lemmata.

{\bf Step 1.}   Define $\widehat{F}:[0,1]\times \K \m \E$ by the formula
\[
\widehat{F} \le t,x \pr := \overline{\mathrm{conv}}\left [ F \le [0,t],x\pr\right ] \text{ for } t\in  [0,1], x\in \K\; (\footnote{Here $F([0,t],x):=F([0,t]\times\{x\})$.}).
\]
\begin{lem}
$\widehat F$ has convex weakly compact values, is H-usc, has sublinear growth and is tangent to $\K$.
\end{lem}
\begin{proof} It is sufficient to show $[0,1]\times \K \ni \le t,x\pr \mm F \le [0,t], x\pr \subset \E$ is H-usc, for the H-upper semicontinuity and other properties of $\widehat F$  follow rather easily by standard arguments. Take $t_0 \in [0,1]$, $x_0 \in \K$ and $\e>0$.
For some  $\delta_0 >0$
\begin{equation}\label{zawieranie-H-ciaglosc-do-lematu}
F\le t, x_0 \pr \subset F \le t_0 , x_0 \pr + B _\E \le 0 , \e /2\pr
\end{equation}
if $t \in \le t_0 - \delta_0, t_0 + \delta_0 \pr \cap [0,1]$.
For every $t\in [0,t_0 + \delta_0 /2]$ there is $\delta \le t\pr = \delta \le t,x_0 \pr>0$ such that
\[
F \le s , x \pr  \subset F \le t,x_0 \pr + B _\E \le 0, \e/2\pr,
\]
provided  $s \in \le t-\delta (t) , t + \delta \le t \pr \pr ,\; x \in B _\K \le x_0 , \delta \le t \pr \pr $.
Let $\left \{ \le t_i - \delta \le t_i\pr,t_i + \delta\le t _i\pr\pr  \right \}_{i=1, \ldots ,k}$ be a finite subcover of an open cover $\left \{ \le t - \delta \le t\pr , t + \delta \le t\pr \pr \right \}_{t \in [0, t_0 + \delta _0/2 ]}$ of $[0,t_0 +\delta_0/2]$.
Put $\delta := \min \left \{ \delta_0 /2, \delta \le t_1 \pr , \ldots, \delta \le t_k \pr \right\}$.\\
\indent Choose $t\in \le t_0 - \delta,t_0 + \delta \pr \cap [0,1]$, $ x\in B_\K \le x_0, \delta\pr$ and let $y\in F \le [0,t ] \times \{ x \}\pr $, i.e.,  $y\in F\le s,x\pr $ for some $s\in [0,t]$. There is $t_i$ such that $s\in \le t_i - \delta \le t_i \pr , t_i + \delta \le t_i \pr \pr$. The inclusion $B_\K \le x_0, \delta \pr \subset B_\K \le x_0 , \delta \le t_i \pr \pr$ implies
\[
y\in F\le s ,x \pr \subset  F \le t_i, x_0 \pr + B _\E \le 0, \e/2 \pr.
\]
If $t_i \leq t_0$, then
\[
y \in F \le t_i, x_0 \pr + B _\E \le 0, \e /2 \pr \subset F \le [0,t_0 ], x_0\pr +B _\E \le 0, \e \pr,
\]
while if  $t_i >t_0, t_i \in [0, t_0 + \delta /2]$, then $t_i - t_0 \leq \delta _0 /2$, so  by \eqref{zawieranie-H-ciaglosc-do-lematu}
$$
y \in F \le t_i , x_0 \pr + B _\E \le 0, \e/2 \pr \subset F \le t_0, x_0 \pr +B _\E \le 0, \e \pr \subset F \le [0,t_0 ],x_0\pr + B _\E \le 0, \e \pr,
$$
i.e., $F([0,t],x)\subset F([0,t_0],x_0)+B_\E(0,\e)$ if $t\in [0,1]$, $|t-t_0|,\delta$ and $x\in\K$, $\|x-x_0\|<\delta$. \end{proof}
Using the same methods as in Lemma \ref{lem-do-konstr-stopnia}  we get:
\begin{lem}\label{lem-istnienie-alfy-T}
There are is $\alpha>0$, $T>0$ such that $0 \notin A x + \widehat F\le T, B_\K \le x, \alpha\pr \pr  + B_\E \le 0, \alpha \pr$ for $x\in \partial U$.\hfill $\square$
\end{lem}
{\bf Step 2.} By Lemma \ref{selekcja-styczna}, there is  locally Lipschitz $f:\K\to \E$ being tangent to $\K$ and an $\alpha$-approximation of $F\le 0,\cdot\pr$.  Arguing as in Lemma \ref{lem-do-konstrukcji-stopnia-2} we find $h_0>0,\, h_0\omega <1$ such that
\begin{equation}\label{inkluzja-retrakcja-brak-pktow-stalych}x\neq r\circ J_h \le x + hf \le x \pr \pr \quad\text{for }x\in \partial U\;\;\hbox{for}\;\; h\in (0,h_0].\end{equation}
Observe that, by definition (see \eqref{definicja-stopnia}),
\begin{equation}\label{deg_pr_st}
\mathrm{deg}_\K \le A + F \le  0,\cdot \pr , U \pr = \mathrm{Ind}_\K \le r\circ J_h \le I + h f\pr, U\pr,\end{equation}
where $r:\E\to\K$ is a retraction. In what follows let $r$ be a {\em metric retraction}, i.e., $\|x-r(x)\|=d(x,\K)$ for any $x\in\E$.\\
\indent Define the auxiliary set-valued map $G:[0,1]\times \K \m \E$ by the formula
\[
G \le z,x \pr:= \le 1-z \pr f\le x \pr + z \widehat{F}\le T, x \pr\;\; z\in [0,1],\;x\in \K.
\]
Obviously, $G$ is H-usc, tangent to $\K$, has sublinear growth and convex weakly compact values.
By Theorem \ref{tw.o-istnieniu}, the solution set of the below problem is $R_\delta$:
\begin{equation}\label{homotopia-z-grubego}
\begin{cases}
\dot{u} = Au + G \le z,x \pr,\;\; u\in \K,\; z \in [0,1] \\
u\left (0\right ) = x \in \mathrm{cl}\, U.
\end{cases}
\end{equation}
\begin{lem}\label{lem-nie-istnienie-rozwiazan-z-pktami-stalymi}
There is $t_0 \in (0, T]$ such that for every $t\in (0,t_0]$  no solution $u$ of \eqref{homotopia-z-grubego} starting at $x\in \partial U$ is such that  $u\le t\pr = x$.
\end{lem}
\begin{proof}
Suppose to the contrary that for each integer $n\geq n_0$, where $n_0^{-1} <T$ there are $x_n \in \partial U$, $t_n \in (0,n^{-1}]$, $z_n\in [0,1]$ and the solution $u_n:[0,t_n]\to \K$ of \eqref{homotopia-z-grubego} such that $u_n \le  0 \pr =x_n = u_n \le t_n \pr $.
Then there is $w_n \in L^1 \le [0,t_n],\E\pr $ such that $w_n \le s \pr \in G\le z_n, u_n \le s \pr \pr$ for a.e. $s\in [0,t_n]$ and
\begin{equation}\label{wzor-wewnatrz-lematu-rozszerzanie-periodyczne}
u_n \le t \pr = S \le t \pr x_n + \int_0^t S \le t -s \pr w_n \le s \pr \mathrm{d}s, \quad t\in  [0,t_n].
\end{equation}
Extending periodically, we may assume that $u_n$ and $w_n$ are defined on $[0,T]$, i.e. $u_n \in \mathcal{C}\le [0,T], \K\pr$, $w_n \in L^1 \le [0,T], \E\pr$.
The semigroup property ensures that formula \eqref{wzor-wewnatrz-lematu-rozszerzanie-periodyczne} is valid for every $t\in [0,T]$ and $w_n \in G \le z_n ,u_n \le s \pr \pr$ for a.e. $s\in [0,T]$.
Thus $u_n$ is a solution on $[0,T]$ of \eqref{homotopia-z-grubego}.\\
\indent The growth condition and Gronwall's inequality imply that $\left \{ u_n \right \}_{n\geq 1}$ is bounded.
Therefore $\left \{ w_n \right \}_{n\geq 1}$ being a.e. bounded by a constant is weakly relatively compact in $L^1 \le [0,T], \E \pr$ (cf. \cite[Cor. 2.6]{diest}).
Passing to a subsequence we may assume that $w_n \rightharpoonup w_0 \in L^1 \le [0,T], \E \pr $ and $z_n \to z_0 \in [0,1]$.\\
\indent To prove that $\left \{ u_n  \right \}_{n\geq 1}$ is relatively compact it is enough to show that  so is $\left \{ x_n \right \}_{n\geq 1}$ (cf. Lemma \ref{uzwarcanie}).
Take $T_0 \in (0,T)$ and put $k_n:= \le [T_0 /t_n ] +1 \pr$.
Then $r_n:= k_n t_n - T_0 \to 0$ and $u_n \le k_n t_n \pr =x_n$ for large $n$.
So for sufficiently large $n\geq 1$: $T_0 +r_n < T$ and
\[
x_n = u_n \le k_n t_n \pr = S \le T_0 + r_n \pr x_n + \int _0 ^{T_0 + r_n } S \le T_0 + r_n -s \pr w_n \le s \pr \mathrm{d}s.
\]
The compactness of the semigroup yields that $\left \{ x_n \right \}_{n\geq 1}$ is relatively compact and so $x_n \to x_0 \in \partial U$.\\
\indent Thus $u_n \to u_0 \in \mathcal{C}\le [0,T], \K\pr$, and by the uniform equicontinuity of $\left \{ u_n \right \}_{n\geq 1}$
\[
\| u_0 \le t \pr - x_0 \| \leq \|u_0 -u_n \|+ \|u_n \le t \pr - u_n  \le [t/t_n ] t_n \pr \|+ \|x_n - x_0\| \to 0,
\]
hence $u\le t \pr = x_0$ for $t\in [0,T]$.
Therefore
\begin{equation}\label{wzor-rozw-stacjonarne}
x_0 = S \le t \pr x_0 + \int _0 ^t S \le t-s \pr w_0 \le s \pr \mathrm{d}z
\end{equation}
and $w_0 \le s \pr \in G \le z_0 , x_0 \pr$ for a.e. $s\in [0,T]$.
Since $[0,T] \ni t \mapsto \int_0^t w_0 \le s \pr \mathrm{d}s $ is a.e. differentiable, take $\xi \in (0,T)$ such that $w_0\le \xi \pr \in G \le z_0 , x_0 \pr $ and $\dfrac{d}{dt} \Big |_{t=\xi} \int_0^t w_0 \le s \pr \mathrm{d}s = w_0 \le \xi \pr $.
By \eqref{wzor-rozw-stacjonarne} for small $\eta >0$; $x_0 = S \le \eta \pr x_0 + \int _\xi ^{\xi + \eta }S \le \xi + \eta - s \pr w_0 \le s \pr \mathrm{d}s.$\\
\indent We show that
\begin{equation}\label{wzor-slaba-zbieznosc}
\frac{1}{\eta} \int _\xi ^{\xi + \eta} \le w_0 \le s \pr - S \le \xi + \eta -s \pr w_0 \le s \pr \pr \mathrm{d}s \rightharpoonup 0\quad\text{as }\eta \to 0^+.
\end{equation}
Take $p\in \E^\ast$ and $\e>0$.
Then
\[
\left\langle \frac{1}{\eta} \int_\xi ^{\xi + \eta}\le w_0 \le s \pr - S \le \xi + \eta - s \pr w_0\le s \pr\pr \mathrm{d}s,\; p \right\rangle = \frac{1}{\eta} \int_\xi ^{\xi + \eta}\left\langle w_0 \le s \pr ,\; p - S^\ast \le \xi + \eta - s \pr p\right\rangle \mathrm{d}s
\]
and since $\E$ is the Hilbert space the dual semigroup $\left \{ S ^\ast \le t\pr  \right \}_{t\geq 0}$ is strongly continuous.
Thus there is $\delta >0$ such that  $\| S^\ast \le t \pr p - p \|  < \e / M$ if $0 \leq t < \delta$ where $M:= \sup _{y \in G \le z_0 ,x_0 \pr } \| y \|$.
If $0<\eta <\delta$, then for $\xi \leq s \leq \xi + \eta$ we have $\| S ^\ast \le \xi + \eta - s \pr p - p \| < \e /M$ so
\[
| \langle w_0 \le s \pr , p - S ^\ast \le \xi + \eta - s \pr p\rangle | < \e\quad \text{for a.e. }s \in [\xi, \xi + \eta ],
\]
what proves \eqref{wzor-slaba-zbieznosc}.
As a result
\[
\frac{S\le \eta \pr x_0 - x_0}{\eta} = \frac{1}{\eta} \int_\xi ^{\xi + \eta}\le w_0 \le s \pr - S \le \xi + \eta - s \pr w_0\le s \pr\pr \mathrm{d}s - \frac{1}{\eta} \int_\xi ^{\xi + \eta} w_0 \le s \pr \mathrm{d}s \rightharpoonup - w_0 \le \xi \pr.
\]
In view of \cite[Th. 2.1.3]{pazy}  $x_0\in D \le A \pr \cap \partial U$ and $A x_0 = -w_0 \le \xi \pr $.
Hence $0 = Ax_0 + w_0 \le \xi \pr \in Ax_0 + G \le z_0 , x_0 \pr $ what contradicts Lemma \ref{lem-istnienie-alfy-T}, since $G\le z_0,x_0\pr \subset \mathrm{conv}F_T \le B_\K \le x_0, \alpha \pr \pr + B_\E \le 0 , \alpha \pr $.
\end{proof}
{\bf Step 3.} Recall now the solution operator $\Sigma : \K \m \mathcal{C}\le [0,1], \K \pr $ (see \eqref{op_rozw}) and the $t$-Poincar\'e operator $\Sigma_t : \K\m \K$ associated with to \eqref{inkluzja-rozn} and consider their restrictions to $\mathrm{cl}\,U$. By a slight abuse of notation, we will still denote these restrictions by the same symbols, i.e. $\Sigma :\mathrm{cl}\,U \m \mathcal{C}\le [0,1] , \K\pr$ is represented by the $c$-admissible pair
\[
\mathrm{cl}\,U \xleftarrow{p_\Sigma} \Gamma \xrightarrow{q_\Sigma} \mathcal{C} \le [0,1],\K\pr
\]
with $p_\Sigma$, $q_\Sigma$ and $\Gamma= \left\{ \le x, u \pr \in \mathrm{cl}\,U\times\mathcal{C}\le [0,1],\K \pr \mid u\in \Sigma \le x\pr \right\}$ having the same sense as in \eqref{op_rozw}, while $\Sigma_t: \mathrm{cl}\,U\m \K$ is represented by
\[
\mathrm{cl}\,U \xleftarrow{p_t} \Gamma \xrightarrow{q_t}\K
\]
with $p_t:=p_\Sigma$, $q_t:= e_t \circ q_\Sigma$.\\
\indent Taking into account \eqref{deg_pr_st} and Remark \ref{uwaga-stopien-party-jednowartosciowej} we are to show that for sufficiently small $t>0$, $h>0$, the $c$-admissible pairs $\le p_t,q_t \pr $ and $\le\mathrm{id}, r\circ J_h \le I + hf\pr \pr$, where $\mathrm{id}$ stand for the identity on $\cl U$, are $c$-homotopic via a compact $c$-homotopy without fixed points on the boundary $\partial U$.  This will be done in several stages.

For any $x\in\K$, the problem
\begin{equation}\label{prob_f}
\begin{cases}
\dot{u} = Au + f(u)\; \text{ for } t \in I,\\
u\left (0\right ) = x
\end{cases}
\end{equation}
possesses the unique solution $P(x)$; the map $P:\mathrm{cl}\,U\to \mathcal{C}\le [0,1], \K\pr$ is continuous. For $t\in [0,1]$, the Poincar\'{e} $t$-operator $P_t : \mathrm{cl}\,U \to \K$ associated to \eqref{prob_f}, i.e., given by $ P_t \le x\pr := P(x) \le t \pr$ for $x\in \mathrm{cl}\,U$, is compact.\\
\indent Let us consider the Poincar\'e operator $\Phi:[0,1]\times \mathrm{cl}\,U \m \mathcal{C}\le [0,1],\K\pr$ associated with the problem
\begin{equation}\label{zag-pocz-t-op-hom}
\dot{u} \in A u + \le 1-z \pr f \le u\pr + z F \le t, u \pr, \text{ for } t\in [0,1],z \in [0,1], u \in \K.
\end{equation}
It is clear that $\Phi$ is cell-like (cf. Remark \ref{uwaga-parametryczna-wersja}).
Fix $t\in [0,1]$ and consider the Poincar\'e $t$-operator $\Phi_t : [0,1] \times \mathrm{cl}\,U \m \K$ defined by
$$ \Phi _t \le z,x \pr := \left \{ u \le t\pr \in \K \mid u \in \Phi \le z,x \pr \right \}.$$
As before $\Phi_t$ is compact and $c$-admissible (cf. \ref{uwaga-c-admisible-solution-mapping}). If $u\in \Phi \le z,x\pr$ for some $z\in [0,1],x\in \K$ then $u\restr {[0,t_0]}$ is also the solution of the problem \eqref{homotopia-z-grubego} on the segment $[0,t_0]$, since $F\le t,y\pr \subset \widehat{F}\le T, y\pr $ for $t\in [0,T], y\in \K$.
Thus, by Lemma \ref{lem-nie-istnienie-rozwiazan-z-pktami-stalymi}
\[
x\notin \Phi _t \le z,x\pr \quad\text{for } t \in [0,t_0], x\in \partial U, z\in [0,1].
\]
Clearly $\Phi \le 1,\cdot \pr = \Sigma$ and $\Phi \le 0,\cdot\pr  = P$ (so $\Phi_t \le 1, \cdot \pr = \Sigma _t $ and $\Phi _t \le 0,\cdot \pr = P_t $).
Therefore, the canonical pair $\le p_\Phi,q_\Phi \pr$ representing $\Phi$ is the $c$-homotopy joining $\le p_\Sigma, q_\Sigma \pr$ to the canonical pair representing $P$. Therefore  the pair $\le p_\Phi , e_t \circ q_\Phi \pr$ representing $\Phi_t$ is a $c$-homotopy joining $\le p_t, q_t \pr $ to $\le \mathrm{id}_{\mathrm{cl}\,U} , P_t \pr$. Hence:
\begin{lem}\label{lem-homotopia-laczaca-t-Poincare-z-prawa-strona}
If $t\in (0,t_0]$ ($t_0$ is given by Lemma \ref{lem-nie-istnienie-rozwiazan-z-pktami-stalymi}) then the pairs $\le p_t , q_t \pr $ and $\le \mathrm{id}_{\mathrm{cl}\,U} , P_t\pr$ are $c$-homotopic via the compact $c$-homotopy without fixed points on $\partial U$\hfill $\square$.
\end{lem}

\begin{prop}\label{rew}There are $0<t_1\leq t_0$ and $0< h_1\leq h_0$ such that for $t\in (0,t_1], h\in (0,h_1]$ maps $P_t$ and $g:=r\circ J_h(I+hf)$ (see \eqref{deg_pr_st}) are homotopic via a compact  homotopy without fixed points on $ \partial U$.
\end{prop}
\begin{proof}\textbf{Claim 1.} For sufficiently small $t > 0$ and $h >0$ the Poincar\'e $t$-operators associated with \eqref{prob_f} and the problem
\[
\le P_{J_h} \pr \; \begin{cases}
\dot{u} = - u + g(u),\\
u\left (0\right ) = x
\end{cases}
\]
are homotopic via a condensing  (with respect to the Hausdorff measure of noncompactness) homotopy without fixed points on $\partial U$.\\
\indent Fix $h\in (0,h_0]$ and consider a parameterized problem
\begin{equation}\label{zag-pocz-g_z}
\dot{u} = z A + g_z \le u \pr,\;\;\text{for }\;z\in [0,1],\; u\in \K
\end{equation}
where $g_z: \K \to \E$ is defined by
\[
g_z \le x \pr  := z f\le x\pr +\le 1-z \pr \le - x + g(x)\pr \quad \text{for }\;z\in [0,1],\;x\in \K.
\] Clearly, for each $z\in [0,1]$, $g_z$ is locally Lipschitz, since so are $f$ and $r$. Moreover, for any $x\in \K$, $f\le x\pr \in T_\K \le x\pr $ and
\[
g(x)=-x + r \circ J_h\le x + hf \le x \pr \pr \in \K-x \subset T_\K \le x\pr\quad\text{for }x\in \K
\]
and, hence, $g_z \le x \pr \in T_\K \le x \pr $ for $x\in \K$. It is easy to see that $g_z$ has sublinear growth and the semigroup $\left \{ S \le zt\pr  \right \}_{t\geq 0}$ generated by the operator $zA$ leaves the set $\K$ invariant.
Thus, for any $z\in [0,1]$, $x\in\K$, the problem \eqref{zag-pocz-g_z} along with the initial condition $u \le 0\pr = x$ has a unique mild solution $\Theta \le x,z\pr : [0,1] \to \K$. Obviously $\Theta(x,0)$ is the solution to $(P_{J_h})$ while $\Theta(x,1)$ is the solution to $(P_{A,f})$.\\
\indent To see that, for some small $t>0$, the map 
$$\cl U\times [0,1]\ni (x,z)\mapsto \Theta_t(x,z):=\Theta(x,z)(t)\in \K$$ is the required homotopy joining the Poincar\'e $t$-operators of $(P_{J_h})$ and $(P_{A,f})$ we need to study a different form of \eqref{zag-pocz-g_z}.  Namely consider the following family
$\left \{ A _z : D \le A \pr \to \E\right\}_{z\in [0,1]}$ of operators
defined by 
$$A_z := \le z - 1 - \frac{z}{h} \pr I +zA\;\; \hbox{for};z\in [0,1]$$ and  let $f_z :\K \to \E$ be given by the formula
\[
f_z \le x \pr := \le \frac{z}{h}I+ \le 1-z\pr r \circ J_h \pr \le x + h f \le x \pr \pr, \quad\text{for }\;h \in (0,h_0],\; z \in [0,1],\; x\in \K.
\]
A straightforward calculation shows that for $z\in [0,1]$ and $x\in \K$
$A_z x + f_z \le x\pr =  zA x + g_z \le x \pr$. Hence and by the use of 
the formula \cite[Chapter 3.1. (1.2)]{pazy} and the Fubini theorem we gather that $\Theta \le x,z\pr $ is also the unique solution to the problem
\begin{equation}\label{zag-pocz-f_z}
\begin{cases}
\dot{u}\le t \pr = A_z u \le t \pr + f_z \le u \le t\pr \pr ,\\
u \le 0 \pr = x.
\end{cases}
\end{equation}
By \cite[Theorem 4.5]{cw}, the operator $\Theta_t$ is continuous and condensing with respect to the Hausdorff measure of noncompactness; moreover there is $t_1^\prime >0$ such that if $t\in (0,t_1^\prime]$, then
\begin{equation}\label{brak-pktow-st-na-brzegu-homotopia}
\Theta _t \le x,z\pr \neq x \quad \text{for }\; z\in [0,1],\; x\in \partial U.
\end{equation}
We have just shown that if $0<t<t_1'$, then the Poincar\'e $t$-operator $P_t$ is homotopic to the Poincar\'e $t$-operator $\Theta_t \le\cdot,0\pr $ associated tp $(P_{J_h})$ via a condensing homotopy without fixed points on  $[0,1]\times \partial U$.\\
\indent \textbf{Claim 2.} For sufficiently small $t>0$ and sufficiently small $h>0$, the Poincar\'e $t$-operator $\Theta_t \le\cdot,0\pr $ associated with $(P_{J_h})$
\[
\dot{u} = - u + g \le u \pr \quad\text{for } u\in \K
\]
is homotopic to $g$ via condensing homotopy without fixed points on $\partial U$.\\
\indent Indeed, for a fixed $t$ and for $x\in\cl U$, $z\in [0,1]$ let 
\[
\widehat{\Psi}_t \le z ,x \pr :=
\begin{cases}
 \le 1 - \frac{1}{z \le t+z -zt \pr} \pr x + \frac{1}{z \le t+z-zt \pr}\Theta_{zt} \le x,0 \pr \quad\text{for }\; z\in (0,1],\\
g \le x \pr \qquad\qquad\qquad\qquad\qquad\qquad\qquad\;\;\;\;\text{for } z=0,
\end{cases}
\]
As in  \cite[Prop. 4.3]{cw} one shows that $\widehat{\Psi}_t$ is continuous and there is $t_1 \in (0,1), t_1 < t_1^\prime$ such that  $\widehat{\Psi}_t$ is condensing and $\widehat{\Psi}_t\le z,x\pr \neq x$ for $z\in [0,1], x\in \partial U$ provided for $t\in (0, t_1]$. \\
\indent Take $0<t\leq t_1$ and let $\Psi_t := r\circ \widehat{\Psi}_t :[0,1] \times \mathrm{cl}\,U \to \K$. Then $\Psi_t$ is continuous and condensing as the superposition of $\widehat{\Psi}_t $ with the nonexpansive metric projection $r$. We shall make of the following general observation.
\begin{lem}\label{lem-o-rzucie}
If $x\in \K$, then $y\in x + \bigcup _{h>0} h \le \K -x \pr $ and $r\le y\pr = x $ if and only if $y=x$.
\end{lem}
\begin{proof} There are $h_0 >0 $ and $k_0 \in \K$ such that $y= x + h_0 \le k_0 -x\pr$. The so-called variational chracterization of $r$ (see e,g. \cite[Th. 5.2]{brezis}) yields that for all $k\in \K$, 
$\langle k - r\le y \pr ,y - r \le y \pr \rangle = \langle k-x, y-x \rangle \leq 0$
Hence $\langle k-x, \le x + h_0 \le k_0 -x \pr \pr -x \rangle = h_0\langle k-x , k_0 - x \rangle \leq 0$ for every $k\in \K$.
Thus, $k_0 = x$ and $y = x + h_0 \le k_0 - x \pr  =x$.
\end{proof}
Observe that  $\Psi _t \le x,0\pr = r\circ  g \le x \pr =r\circ J_h \le x + h f \le x \pr \pr$ and, in view of \eqref{inkluzja-retrakcja-brak-pktow-stalych}, $\Psi _t \le x,0\pr \neq x$ for $x\in \partial U$. If $z\in (0,1]$, then for $x\in\cl U$
\[
\widehat{\Psi}_t \le x,z\pr = x + \frac{1}{z \le t+z -zt \pr}\le \Theta_{zt} \le x,0 \pr - x \pr \in x +\bigcup_{h\geq 0}h \le \K - x \pr.
\]
For such $z$ and $x$, by Lemma \ref{lem-o-rzucie}, $x = \Psi_t \le s ,x \pr = r\circ \widehat{\Psi}_t \le s,x\pr $ if and only if $x = \widehat{\Psi}_t \le s,x\pr$. Therefore, in view of \cite[Prop. 4.3., Claim 2.]{cw}, $\Psi _t \le s,x \pr \neq x $ for $s\in (0,1], x\in \partial U$.
As a result: $\Psi_t \le x,z\pr \neq x$ for $z\in [0,1]$, $x\in \partial U$ provided  $t\in (0,t_1]$.
\indent Finally. in order to obtain a compact homotopy joining $g$ to $\Theta(\cdot,0)$ we will rely on the following result.
\begin{lem}\label{lem-zwarta-zamiast-kondensujacej}\cite[Th. 3.1.4.]{meas-noncomp}, \cite[Def. 3.1.7.]{meas-noncomp}
Let $X\subset \K$ be bounded closed and $f_0,f_1:X\to\K$ be compact maps.
If $h: [0,1]\times X \to \K$ is a condensing homotopy joining $f_0$ to $f_1$, then there is the compact homotopy $H:[0,1]\times X \to\K$ joining $f_0$ to $f_1$ having the same fixed points as $h$ does.\hfill $\square$\end{lem}

This establishes  Proposition \ref{rew}, since Lemma \ref{lem-zwarta-zamiast-kondensujacej} produces a compact homotopy out of $\Psi_t$ (recall that $g$ and $P_t$ are compact).\end{proof}

\noindent To sum up, we proved that for sufficiently small $t>0$ and $h>0$:
\begin{enumerate}
\item the $c$-admissible pair $\le p_t , q_t \pr $ is $c$-homotopic to the pair $\le \mathrm{id}_{\mathrm{cl}\,U},P_t \pr$ via the compact $c$-homotopy without fixed point on $[0,1]\times \partial U$ (cf. \ref{lem-homotopia-laczaca-t-Poincare-z-prawa-strona});
\item the Poincar\'{e} $t$-operator $P_t:\mathrm{cl}\,U \to \K$ is homotopic to $r\circ J_h \le I + hf \pr : \mathrm{cl}\,U \to \K$ via the compact homotopy without fixed points on $\partial U$ (cf. Corollary \ref{wniosek-zamiana-kondensujacej-zwarta}).
\end{enumerate}
Thus, in view of \ref{uwaga-stopien-party-jednowartosciowej} we have
\[
\mathrm{Ind}_\K \le\le p_t , q_t \pr , U \pr = \Ind_\K \le \le \mathrm{id}_{\mathrm{cl}\,U},P_t \pr ,U \pr = \Ind_\K \le P_t , U \pr = \mathrm{deg}_\K \le A + F \le 0 , \cdot \pr , U \pr .
\]
This concludes the proof of Theorem \ref{glowne}.\hfill $\square$

Let us finally formulate a direct single-valued counterpart of this result being a direct generalization of \cite{cw}. 

\begin{cor}
Assume that $A$ and $U$ are the same as in Theorem \ref{glowne}.
Additionally, let $f:[0,1]\times \K \to \E$ be tangent to $\K$ locally Lipschitz function with sublinear growth.
If $0\neq Ax + f \le 0 , x \pr $, $x\in \partial U$, then there is $t_0 \in (0,1]$ such that for every $t\in (0,t_0]$
\[\mathrm{Ind}_\K \le P_t, U \pr = \mathrm{deg}_\K \le A +f \le 0, \cdot \pr , U \pr,\]
where $P_t : \mathrm{cl}\,U\to \K$ is the Poincar\'e $t$-operator associated with the problem $\dot{u}\le t \pr =A u\le t\pr  + f \le t, u \pr $.
\end{cor}

\end{document}

\textbf{Step 2.} It is enough to show that if $H:[0,1] \times \K \m \E$  is H-usc, then so is $\overline{\mathrm{conv}}H$.
Take $t_0,\in [0,1], x_0 \in \K $ and $\e>0$.
By the assumption there is $\delta >0$ such that
\[
H\le t,x\pr \subset H \le t_0,x_0\pr + B_\E \le 0 ,\e /2 \pr\subset \mathrm{conv}H \le t_0 ,x_0\pr + B_\E \le 0, \e/2 \pr.
\]
Due to the convexity of the right-hand side we get
\[
\mathrm{conv}H \le t ,x \pr \subset \mathrm{conv}H \le t_0 ,x_0\pr + B_\E \le 0, \e /2 \pr,
\]
Thus, for $t\in \le t_0 - \delta , t_0 + \delta \pr \cap [0,1], x \in B_\K \le 0, \delta \pr$ we get
\[
\overline{\mathrm{conv}}H \le t, x \pr \subset \mathrm{cl}\left [ \mathrm{conv}H \le t_0 ,x_0\pr + B_\E \le 0, \e /2 \pr \right ] \subset \overline{\mathrm{conv}}H \le t_0 ,x_0\pr + B_\E \le 0, \e \pr.
\qedhere
\]

Note that $\widehat{F}$ also satisfies $\le F_ 4 \pr $ and $\le F_5\pr$.
What is more, $\widehat{F}$ has convex weakly compact values.
Indeed, convexity is obvious and since $F$ maps compact sets onto weakly compact (cf. \ref{slabo-usc}), the set $F \le [0,t]\times \{ x \} \pr $ is weakly compact for every $t\in [0,1], x\in \K$.
By the Krein---\v{S}mulian theorem the set $\overline{\mathrm{conv}}F\le [0,t]\times \{x\} \pr  = \widehat{F }\le t,x \pr$ is weakly compact. \